\documentclass[a4paper, intlimits, reqno]{amsart}

\usepackage[english]{babel}
\usepackage[T1]{fontenc}
\usepackage[utf8]{inputenc}

\usepackage{amsmath}
\usepackage{amssymb}
\usepackage{MnSymbol}
\usepackage{amsthm}
\allowdisplaybreaks
\usepackage{amsfonts}
\usepackage{mathrsfs} 
\usepackage{mathtools}
\usepackage{nicefrac}
\usepackage{enumitem}
\usepackage{multicol}
\usepackage{url}
\usepackage{dsfont}
\usepackage[numbers]{natbib}
\usepackage{prettyref}
\usepackage{doi}
\usepackage{orcidlink}

\newrefformat{def}{Definition \ref{#1}}
\newrefformat{rem}{Remark \ref{#1}}
\newrefformat{sect}{Section \ref{#1}}
\newrefformat{prop}{Proposition \ref{#1}}
\newrefformat{thm}{Theorem \ref{#1}}
\newrefformat{cor}{Corollary \ref{#1}}
\newrefformat{ex}{Example \ref{#1}}
\newrefformat{que}{Question \ref{#1}}

\swapnumbers
\newtheoremstyle{dotless}{}{}{\itshape}{}{\bfseries}{}{}{}
\theoremstyle{dotless}
\theoremstyle{plain}
\newtheorem{thm}{Theorem}[section]

\newtheorem{prop}[thm]{Proposition}
\newtheorem{cor}[thm]{Corollary}
\theoremstyle{definition}
\newtheorem{defn}[thm]{Definition}
\newtheorem{rem}[thm]{Remark}

\newtheorem{que}[thm]{Question}

\newcommand{\N} {\mathbb{N}}
\newcommand{\R} {\mathbb{R}}
\newcommand{\C} {\mathbb{C}}
\newcommand{\K} {\mathbb{K}}
\newcommand{\D} {\mathbb{D}}
\newcommand{\F} {\mathcal{F}(\Omega)}
\newcommand{\FE} {\mathcal{F}(\Omega,E)}
\newcommand{\FV} {\mathcal{F}\nu(\Omega)}
\newcommand{\FVE} {\mathcal{F}\nu(\Omega,E)}
\newcommand{\Feps} {\mathcal{F}_{\varepsilon}\nu(\Omega,E)}
\newcommand{\acx} {\operatorname{acx}}
\newcommand{\oacx} {\overline{\operatorname{acx}}}

\DeclareMathOperator{\id}{id}
\providecommand{\differential}{\mathrm{d}}
\renewcommand{\d}{\differential}

\begin{document}

\title[Extension]{Extension of vector-valued functions and weak-strong principles for differentiable functions of finite order}
\author[K.~Kruse]{Karsten Kruse\,\orcidlink{0000-0003-1864-4915}}
\address{Hamburg University of Technology\\ Institute of Mathematics \\
Am Schwarzenberg-Campus~3 \\
21073 Hamburg \\
Germany}
\email{karsten.kruse@tuhh.de}



\date{\today}
\begin{abstract}
In this paper we study the problem of extending functions with values in a locally convex Hausdorff space $E$ 
over a field $\K$, which have weak extensions in a weighted Banach space $\mathcal{F}\nu(\Omega,\K)$ of 
scalar-valued functions on a set $\Omega$, to functions in a vector-valued counterpart $\FVE$ 
of $\mathcal{F}\nu(\Omega,\K)$. Our findings rely on a description of vector-valued functions 
as linear continuous operators and extend results of Frerick, Jord\'{a} and Wengenroth. 
As an application we derive weak-strong principles for continuously partially differentiable functions of finite order, 
vector-valued versions of Blaschke's convergence theorem for several spaces and Wolff type descriptions of dual spaces.
\end{abstract}
\maketitle
\section{Introduction}
This paper centres on the problem of extending a vector-valued function $f\colon \Lambda\to E$ from a 
subset $\Lambda\subset\Omega$ to a locally convex Hausdorff space $E$ if the scalar-valued functions 
$e'\circ f$ are extendable for each $e'\in G\subset E'$ under 
the constraint of preserving the properties, like holomorphy, of the scalar-valued extensions. 
This problem was considered, among others, by Grothendieck \cite{Grothendieck1953,Gro}, Bierstedt \cite{B2}, 
Gramsch \cite{Gramsch1977}, Grosse-Erdmann \cite{grosse-erdmann1992,grosse-erdmann2004}, 
Arendt and Nikolski \cite{Arendt2016,Arendt2000,Arendt2006}, 
Bonet, Frerick, Jord\'a and Wengenroth \cite{B/F/J,F/J,F/J/W,jorda2005,jorda2013} and us \cite{kruse2018_3}.
 
Often, the underlying idea to prove such an extension theorem is to use a representation of an $E$-valued function by 
a continuous linear operator. Namely, if $\F:=\mathcal{F}(\Omega,\K)$ is a locally convex Hausdorff space 
of scalar-valued functions on a set $\Omega$ such that the point evaluations $\delta_{x}$ at $x$ belong to 
the dual $\F'$ for each $x\in\Omega$, then the function $S(u)\colon\Omega\to E$ given by
$x\mapsto u(\delta_{x})$ is well-defined for every element $u$ of Schwartz' $\varepsilon$-product 
$\F\varepsilon E:=L_{e}(\F_{\kappa}',E)$ where the dual $\F'$ is equipped with the topology of uniform convergence 
on absolutely convex compacts subsets of $\F$, the space of continuous linear operators $L(\F_{\kappa}',E)$ 
is equipped with the topology of uniform convergence on the equicontinuous subsets of $\F_{\kappa}'$ 
and $E$ is a locally convex Hausdorff space over the field $\K$. 
In many cases the function $S(u)$ inherits properties of the functions in $\F$, e.g.\ 
if $\F=(\mathcal{O}(\Omega),\tau_{co})$ is the space of holomorphic functions on an open set $\Omega\subset\C$ 
equipped with the compact-open topology $\tau_{co}$, then the space of functions of the form $S(u)$ with 
$u\in(\mathcal{O}(\Omega),\tau_{co})\varepsilon E$ coincides with the space 
$\mathcal{O}(\Omega,E)$ of $E$-valued holomorphic functions if $E$ is locally complete. Even more is true, namely, 
that the map $S\colon(\mathcal{O}(\Omega),\tau_{co})\varepsilon E\to (\mathcal{O}(\Omega,E),\tau_{co})$ 
is a (topological) isomorphism (see \cite[p.\ 232]{B/F/J}). 
So suppose that there is a locally convex Hausdorff space $\FE$ 
of $E$-valued functions on $\Omega$ such that the map $S\colon \F\varepsilon E\to\FE$ is well-defined 
and at least a (topological) isomorphism into, i.e.\ to its range. 
The precise formulation of the extension problem from the beginning is the following question.

\begin{que}\label{que:weak_strong}
Let $\Lambda$ be a subset of $\Omega$ and $G$ a linear subspace of $E'$.
Let $f\colon\Lambda\to E$ be such that for every $e'\in G$,
the function $e'\circ f\colon\Lambda\to \K$ has an extension in $\F$.
When is there an extension $F\in\FE$ of $f$, i.e.\ $F_{\mid \Lambda}=f$\ ?
\end{que}

Even the case $\Lambda=\Omega$ is interesting because then the question is about properties of vector-valued functions 
and a positive answer is usually called a weak-strong principle. 
From the connection of $\F\varepsilon E$ and $\FE$ it is evident to seek for extension theorems for 
vector-valued functions by extension theorems for continuous linear operators. 
In this way many of the extension theorems of the aforementioned references are derived but in most of the cases 
the space $\F$ has to be a semi-Montel (see \cite{Gramsch1977,Gro,kruse2018_3}) or even a 
Fr\'echet--Schwartz space (see \cite{B/F/J,F/J,Gramsch1977,grosse-erdmann1992,grosse-erdmann2004,Grothendieck1953,jorda2005,kruse2018_3}) 
or $E$ is restricted to be a semi-Montel space (see \cite{B2,kruse2018_3}). 
The restriction to semi-Montel spaces $\F$ resp.\ $E$, i.e.\ to locally convex spaces in which every bounded set 
is relatively compact, is quite natural due to the topology of the dual $\F_{\kappa}'$ in the $\varepsilon$-product 
$\F\varepsilon E$ and its symmetry $\F\varepsilon E\cong E\varepsilon\F$.

In the present paper we treat the case of a Banach space which we denote by $\FV$ 
because its topology is induced by a weight $\nu$. 
We use the methods developped in \cite{F/J/W} and \cite{jorda2013} where, in particular, the special case that 
$\FV$ is the space of bounded smooth functions on an open set $\Omega\subset\R^{d}$ in the kernel of a 
hypoelliptic linear partial differential operator resp.\ a weighted space of holomorphic functions on 
an open subset $\Omega$ of a Banach space is treated. The lack of compact subsets of the infinite dimensional 
Banach space $\FV$ is compensated in \cite{F/J/W} and \cite{jorda2013} by using an auxiliary locally convex 
Hausdorff space $\F$ of scalar-valued functions on $\Omega$ 
such that the closed unit ball of $\FV$ is compact in $\F$. They share the property that 
$S\colon \FV\varepsilon E\to\FVE$ and $S\colon \F\varepsilon E\to\FE$ are topological isomorphisms into 
but usually it is only known in the latter case that $S$ is surjective as well 
under some mild completeness assumption on $E$.
For instance, if $\FVE:=H^{\infty}(\Omega,E)$ is the space of bounded holomorphic 
functions on an open set $\Omega\subset\C$ with values in a locally complete space $E$, then the space 
$\FE:=(\mathcal{O}(\Omega,E),\tau_{co})$ is used in \cite{F/J/W}.
 
Let us outline the content of our paper. We give a general approach to the extension problem for Banach function 
spaces $\FV$. It combines the methods of \cite{F/J/W,jorda2013} with the ones of \cite{kruse2017} 
as in \cite{kruse2018_3} which require that the spaces $\FV$ and $\FVE$ have a certain structure 
(see \prettyref{def:standard_space}).
To answer \prettyref{que:weak_strong} we have to balance the sets $\Lambda\subset\Omega$ and the spaces $G\subset E'$. 
If we choose $\Lambda$ to be `thin', then $G$ has to be `thick' (see Section 3 and 5) and vice versa (see Section 4). 
In Section 6 we use the results of Section 3 to derive and improve weak-strong principles for differentiable functions 
of finite order. Section 7 is devoted to vector-valued Blaschke theorems and Section 8 to Wolff type descriptions 
of the dual of $\F$.
\section{Notation and Preliminaries}
We use essentially the same notation and preliminaries as in \cite[Section 2]{kruse2018_3}.
We equip the spaces $\R^{d}$, $d\in\N$, and $\C$ with the usual Euclidean norm $|\cdot|$.
By $E$ we always denote a non-trivial locally convex Hausdorff space over the field 
$\K=\R$ or $\C$ equipped with a directed fundamental system of 
seminorms $(p_{\alpha})_{\alpha\in \mathfrak{A}}$ and, in short, we write $E$ is an lcHs. 
If $E=\K$, then we set $(p_{\alpha})_{\alpha\in \mathfrak{A}}:=\{|\cdot|\}.$ 
For more details on the theory of locally convex spaces see \cite{F/W/Buch,Jarchow,meisevogt1997}.

By $X^{\Omega}$ we denote the set of maps from a non-empty set $\Omega$ to a non-empty set $X$
and by $L(F,E)$ the space of continuous linear operators from $F$ to $E$ 
where $F$ and $E$ are locally convex Hausdorff spaces. 
If $E=\K$, we just write $F':=L(F,\K)$ for the dual space and $G^{\circ}$ for the polar set of $G\subset F$. 
We write $F\cong E$ if $F$ and $E$ are (linearly topologically) isomorphic. 
We denote by $L_{t}(F,E)$ the space $L(F,E)$ equipped with the locally convex topology $t$ of uniform convergence 
on the finite subsets of $F$ if $t=\sigma$, on the absolutely convex, compact subsets of $F$ if $t=\kappa$ 
and on the bounded subsets of $F$ if $t=b$. We use the symbol $t(F',F)$ for the corresponding topology on $F'$. 
A linear subspace $G$ of $F'$ is called separating if $f'(x)=0$ for every $f'\in G$ implies $x=0$. 
This is equivalent to $G$ being $\sigma(F',F)$-dense (and $\kappa(F',F)$-dense) in $F'$ by the bipolar theorem. 
Further, for a disk $D\subset F$, i.e.\ a bounded, absolutely convex set, 
the vector space $F_{D}:=\bigcup_{n\in\mathbb{N}}nD$ becomes a normed space if it is equipped with 
gauge functional of $D$ as a norm (see \cite[p.\ 151]{Jarchow}). The space $F$ is called locally 
complete if $F_{D}$ is a Banach space for every closed disk $D\subset F$ (see \cite[10.2.1 Proposition, p.\ 197]{Jarchow}).

Furthermore, we recall the definition of continuous partial differentiability of a vector-valued function 
that we need in many examples, especially, for the weak-strong principle for differentiable functions 
of finite order in Section 6. A function $f\colon\Omega\to E$ on an open set $\Omega\subset\R^{d}$ 
to an lcHs $E$ is called continuously partially differentiable ($f$ is $\mathcal{C}^{1}$) 
if for the $n$-th unit vector $e_{n}\in\R^{d}$ the limit
\[
(\partial^{e_{n}})^{E}f(x)
:=\lim_{\substack{h\to 0\\ h\in\R, h\neq 0}}\frac{f(x+he_{n})-f(x)}{h}
\]
exists in $E$ for every $x\in\Omega$ and $(\partial^{e_{n}})^{E}f$ is continuous on $\Omega$ 
($(\partial^{e_{n}})^{E}f$ is $\mathcal{C}^{0}$) for every $1\leq n\leq d$.
For $k\in\N$ a function $f$ is said to be $k$-times continuously partially differentiable 
($f$ is $\mathcal{C}^{k}$) if $f$ is $\mathcal{C}^{1}$ and all its first partial derivatives are $\mathcal{C}^{k-1}$.
A function $f$ is called infinitely continuously partially differentiable ($f$ is $\mathcal{C}^{\infty}$) 
if $f$ is $\mathcal{C}^{k}$ for every $k\in\N$.
For $k\in\N_{0}\cup\{\infty\}$ the linear space of all functions $f\colon\Omega\to E$ which are $\mathcal{C}^{k}$ 
is denoted by $\mathcal{C}^{k}(\Omega,E)$. 
Let $f\in\mathcal{C}^{k}(\Omega,E)$. For $\beta\in\N_{0}^{d}$ with 
$|\beta|:=\sum_{n=1}^{d}\beta_{n}\leq k$ we set 
$(\partial^{\beta_{n}})^{E}f:=f$ if $\beta_{n}=0$, and
\[
(\partial^{\beta_{n}})^{E}f
:=\underbrace{(\partial^{e_{n}})^{E}\cdots(\partial^{e_{n}})^{E}}_{\beta_{n}\text{-times}}f
\]
if $\beta_{n}\neq 0$ as well as 
\[
(\partial^{\beta})^{E}f
:=(\partial^{\beta_{1}})^{E}\cdots(\partial^{\beta_{d}})^{E}f.
\]
If $E=\K$, we usually write $\partial^{\beta}f:=(\partial^{\beta})^{\K}f$. 
We denote by $\tau_{\mathcal{C}^{k}}$ the usual topology on $\mathcal{C}^{k}(\Omega,E)$, namely, 
the locally convex topology given by the seminorms
\[
|f|_{K,m,\alpha}:=\sup_{\substack{x\in K\\ \beta\in\N_{0}^{d},|\beta|\leq m}}p_{\alpha}\bigl((\partial^{\beta})^{E}f(x)\bigr), 
\quad f\in\mathcal{C}^{k}(\Omega,E),
\]
for $K\subset\Omega$ compact, $m\in\N_{0}$, $m\leq k$, and $\alpha\in\mathfrak{A}$.

In addition, we use the following notion for the relation between the $\varepsilon$-product $\F\varepsilon E$ 
and the space $\FE$ of vector-valued functions that has already been described in the introduction. 

\begin{defn}[{$\varepsilon$-into-compatible, \cite[2.1 Definition, p.\ 4]{kruse2018_3}}]
Let $\Omega$ be a non-empty set and $E$ an lcHs. Let $\F\subset\K^{\Omega}$ and $\FE\subset E^{\Omega}$ 
be lcHs such that $\delta_{x}\in\F'$ for all $x\in\Omega$. We call the spaces $\F$ and $\FE$ $\varepsilon$\emph{-into-compatible} if the map
\[
S\colon \F\varepsilon E\to \FE,\;u\longmapsto [x\mapsto u(\delta_{x})],
\]
is a well-defined isomorphism into. We call $\F$ and $\FE$ $\varepsilon$\emph{-compatible} if $S$ is an isomorphism. 
If we want to emphasise the dependency on $\F$, we write $S_{\F}$ instead of $S$.
\end{defn}

\begin{defn}[{strong, consistent}]\label{def:cons_strong}
Let $\Omega$ and $\omega$ be non-empty sets and $E$, $\F\subset\K^{\Omega}$ and $\FE\subset E^{\Omega}$ 
be lcHs. Let $\delta_{x}\in\F'$ for all $x\in\Omega$ and $T^{\K}\colon\F\to \K^{\omega}$ 
and $T^{E}\colon\FE\to E^{\omega}$ be linear maps.
\begin{enumerate}
\item[a)] We call $(T^{E},T^{\K})$ a \emph{consistent} family for $(\mathcal{F},E)$ if 
we have for every $u\in\F\varepsilon E$ that $S(u)\in\FE$ and
\[
\forall\;x\in\omega:\;T^{\K}_{x}:=\delta_{x}\circ T^{\K}\in\F' \quad\text{and}\quad T^{E}S(u)(x)=u(T^{\K}_{x}).
\]
\item[b)] We call $(T^{E},T^{\K})$ a \emph{strong} family for $(\mathcal{F},E)$ if 
we have for every $e'\in E'$, $f\in\FE$ that $e'\circ f\in\F$ and 
\[
\forall\;x\in\omega:\;T^{\K}(e'\circ f)(x)=(e'\circ T^{E}(f))(x).
\]
\end{enumerate}
\end{defn}

This is a special case of \cite[2.2 Definition, p.\ 4]{kruse2018_3} where the considered family $(T^{E}_{m},T^{\K}_{m})_{m\in M}$ only consists 
of one pair, i.e.\ the set $M$ is a singleton. In the introduction we have already hinted that 
the spaces $\F$ and $\FE$ for which we want to prove extension theorems need to have a certain structure, namely, 
the following one. 

\begin{defn}[{generator}]\label{def:standard_space}
Let $\Omega$ and $\omega$ be non-empty sets, $\nu\colon\omega\to (0,\infty)$, $\FE$ a linear subspace of $E^{\Omega}$
and $T^{E}\colon \FE\to E^{\omega}$ a linear map. 
We define the space
\[
\FVE:=\bigl\{f\in \FE\;|\; 
 \forall\;\alpha\in \mathfrak{A}:\; |f|_{\alpha}<\infty\bigr\}
\]
where 
\[
|f|_{\alpha}:=\sup_{x \in \omega}p_{\alpha}\bigl(T^{E}(f)(x)\bigr)\nu(x).
\]
Further, we call $(T^{E},T^{\K})$ the \emph{generator} for $(\FV,E)$, 
in short $(\mathcal{F}\nu,E)$. We write $\FV:=\mathcal{F}\nu(\Omega,\K)$ and omit the index $\alpha$ 
if $E$ is a normed space. If we want to emphasise dependencies, we write $|f|_{\FV,\alpha}$ instead of $|f|_{\alpha}$.
\end{defn}

This is a special case of \cite[Definition 3, p.\ 1515]{kruse2017} where the family of weights only consists 
of one weight function. 
For instance, if $\Omega:=\omega$, $T^{E}:=\id_{E^{\Omega}}$ and $\nu:=1$ on $\Omega$, then 
$\FVE$ is the linear subspace of $\FE$ consisting of bounded functions, in particular, if $\Omega\subset\C$ is open and 
$\FE:=\mathcal{O}(\Omega,E)$, then $\FVE=H^{\infty}(\Omega,E)$ is the space of $E$-valued bounded 
holomorphic functions on $\Omega$.
Due to $(E, (p_{\alpha})_{\alpha\in\mathfrak{A}})$ being an lcHs with directed system of seminorms the 
topology of $\FVE$ generated by $(|\cdot|_{\alpha})_{\alpha\in\mathfrak{A}}$ is locally convex 
and the system $(|\cdot|_{\alpha})_{\alpha\in\mathfrak{A}}$ is directed but need not be Hausdorff. 

\begin{prop}\label{prop:mingle-mangle}
Let $\F$ and $\FE$ be $\varepsilon$-into-compatible, 
$(T^{E},T^{\K})$ a consistent family for $(\mathcal{F},E)$ and a generator for $(\mathcal{F}\nu,E)$ and
the map $i\colon\FV\to\F$, $f\mapsto f$, continuous. We set 
\[
 \Feps:=S(\{u\in\F\varepsilon E\;|\;u(B_{\mathcal{F}\nu(\Omega)}^{\circ \F'})\;\text{is bounded in}\; E\})
\]
where $B_{\FV}^{\circ \F'}:=\{y'\in\F'\;|\;\forall\;f\in B_{\FV}:\;|y'(f)|\leq 1\}$ 
and $B_{\FV}$ is the closed unit ball of $\FV$.
Then the following holds. 
\begin{enumerate}
\item[a)] $\FV$ is Hausdorff and $\delta_{x}\in\FV'$ for all $x\in\Omega$.
\item[b)] Let $u\in\F\varepsilon E$. Then 
\[
\sup_{y'\in B_{\FV}^{\circ \F'}}p_{\alpha}(u(y'))=|S(u)|_{\FV,\alpha},\quad\alpha\in\mathfrak{A}.
\]
In particular, 
\[
 \Feps=S(\{u\in\F\varepsilon E\;|\;\forall\;\alpha\in\mathfrak{A}:\;|S(u)|_{\FV,\alpha}<\infty\}).
\]
\item[c)] $S(\FV\varepsilon E)\subset\Feps\subset\FVE$ as linear spaces. If $\F$ and $\FE$ are even 
$\varepsilon$-compatible, then $\Feps=\FVE$.
\item[d)] If $\FVE$ is Hausdorff and $(T^{E},T^{\K})$ is a strong family for $(\mathcal{F},E)$, then 
\begin{enumerate}
\item[(i)] $(T^{E},T^{\K})$ is a strong, consistent generator for $(\mathcal{F}\nu,E)$ and
\item[(ii)] $\FV$ and $\FVE$ are $\varepsilon$-into-compatible.
\end{enumerate}
\end{enumerate}
\end{prop}
\begin{proof}
Part a) follows from the continuity of $i$ and the $\varepsilon$-into-compatibility of $\F$ and $\FE$. 
Let us turn to part b). As in \cite[Lemma 7, p.\ 1517]{kruse2017} it follows from the bipolar theorem that 
\[
 B_{\FV}^{\circ \F'}=\oacx\{T^{\K}_{x}(\cdot)\nu(x)\;|\;x\in\omega\},
\]
where $\oacx$ denotes the closure w.r.t.\ $\kappa(\F',\FV)$ of the absolutely convex hull $\acx$ of the set 
$D:=\{T^{\K}_{x}(\cdot)\nu(x)\;|\;x\in\omega\}$ on the right-hand side, 
and that 
\begin{align*}
  \sup_{y'\in B_{\FV}^{\circ \F'}}p_{\alpha}(u(y'))
&=\sup_{y'\in\acx(D)}p_{\alpha}(u(y')) 
 =\sup_{y'\in D}p_{\alpha}(u(y'))=\sup_{x\in\omega}p_{\alpha}\bigl(u(T^{\K}_{x})\bigr)\nu(x)\\
&=\sup_{x\in\omega}p_{\alpha}\bigl(T^{E}(S(u))(x)\bigr)\nu(x)
 =|S(u)|_{\FV,\alpha}
\end{align*}
by consistency, which proves part b).

Let us address part c). The continuity of $i$ implies 
the continuity of the inclusion $\FV\varepsilon E\hookrightarrow\F\varepsilon E$ and 
thus we obtain $u_{\mid \F'}\in \F\varepsilon E$ for every $u\in\FV\varepsilon E$.
If $u\in\FV\varepsilon E$ and $\alpha\in\mathfrak{A}$, 
then there are $C_{0},C_{1}>0$ and an absolutely convex compact set $K\subset\FV$ 
such that $K\subset C_{1}B_{\FV}$ and 
\[
 \sup_{y'\in B_{\FV}^{\circ \F'}}p_{\alpha}(u(y'))
 \leq C_{0}\sup_{y'\in B_{\FV}^{\circ \F'}}\sup_{f\in K}|y'(f)|
 \leq C_{0}C_{1},
\]
which implies $S(\FV\varepsilon E)\subset\Feps$. 
If $f:=S(u)\in\Feps$ and $\alpha\in\mathfrak{A}$, then $S(u)\in\FE$ and
\[
 |f|_{\FV,\alpha}=\sup_{x\in\omega}p_{\alpha}(u(T^{\K}_{x})\nu(x))<\infty
\]
by consistency, yielding $\Feps\subset\FVE$. If $\F$ and $\FE$ are even 
$\varepsilon$-compatible, then $S(\F\varepsilon E)=\FE$, which yields $\Feps=\FVE$ by part b).

Let us turn to part d). We have $u_{\mid \F'}\in \F\varepsilon E$ for every $u\in\FV\varepsilon E$ and 
\[
S_{\FV}(u)(x)=u(\delta_{x})=u_{\mid \F'}(\delta_{x})
=S_{\F}(u_{\mid \F'})(x),\quad x\in\Omega.
\]
In combination with $S(\F\varepsilon E)\subset\FE$ and the consistency of $(T^{E},T^{\K})$ for $(\mathcal{F},E)$ 
this yields that $(T^{E},T^{\K})$ is a consistent generator for $(\mathcal{F}\nu,E)$ 
by \cite[Lemma 7, p.\ 1517]{kruse2017} as $\FV$ is a $\operatorname{dom}$-space in the sense of 
\cite[Definition 4, p.\ 1515]{kruse2017} by part a).
The inclusion $\FVE\subset\FE$ and $(T^{E},T^{\K})$ being a strong family for $(\mathcal{F},E)$ imply that
$T^{\K}(e'\circ f)(x)=e'\circ T^{E}(f)(x)$ for all $e'\in E'$, $f\in\FVE$ and $x\in\omega$. 
It follows that $e'\circ f\in\FV$ for all $e'\in E'$ and so that
$(T^{E},T^{\K})$ is a strong generator for $(\mathcal{F}\nu,E)$. 
Thus part (i) holds and implies part (ii) by the first part of the proof of \cite[Theorem 14, p.\ 1524]{kruse2017}. 
\end{proof}

The canonical situation in part c) is that $\Feps$ and $\FVE$ coincide as linear spaces for 
locally complete $E$ as we will encounter in the forthcoming examples, 
e.g.\ if $\FVE:=H^{\infty}(\Omega,E)$ and $\FE:=(\mathcal{O}(\Omega,E),\tau_{co})$ for an open set $\Omega\subset\C$. 
That all three spaces in part c) coincide is usually only guaranteed by \cite[Theorem 14 (ii), p.\ 1524]{kruse2017} 
if $E$ is a semi-Montel space. Therefore the `mingle-mangle' space $\Feps$ is a good replacement for 
$S(\FV\varepsilon E)$ for our purpose. 
\section{Extension of vector-valued functions}
In this section the sets from which we want to extend our functions are `thin'. They are so-called sets of uniqueness. 

\begin{defn}[{set of uniqueness}]
Let $\FV$ be Hausdorff. A set $U\subset\omega$ is called 
a \emph{set of uniqeness} for $(T^{\K},\mathcal{F}\nu)$ if
\[
\forall\; f\in\FV: \;T^{\K}(f)(x)=0 \;\;\forall\,x\in U\;\;\Rightarrow\;\; f=0.
\]
\end{defn}

This definition is a special case of \cite[3.1 Definition, p.\ 8]{kruse2018_3} because
$T^{\K}_{x}\in \FV'$ for all $x\in\omega$ by \cite[Remark 5, p.\ 1516]{kruse2017}.
The span of $\{T^{\K}_{x}\;|\;x\in U\}$ is
weak$^{\ast}$-dense in $\FV'$ by the bipolar theorem if $U$ is a set of uniqueness for $(T^{\K},\mathcal{F}\nu)$. 
The set $U:=\omega$ is always a set of uniqueness for $(T^{\K},\mathcal{F}\nu)$ as $\FV$ is an lcHs 
by assumption. Next, we introduce the notion of a restriction space which is a special case of \cite[3.3 Definition, p.\ 8]{kruse2018_3}.

\begin{defn}[{restriction space}]
Let $G\subset E'$ be a separating subspace and $U$ a set of uniqueness for $(T^{\K},\mathcal{F}\nu)$.
Let $\mathcal{F}\nu_{G}(U,E)$ be the space of functions $f\colon U\to E$ such that for every $e'\in G$ there is 
$f_{e'}\in\FV$ with $T^{\K}(f_{e'})(x)=(e'\circ f)(x)$ for all $x\in U$.
\end{defn}

The time has come to use our auxiliary spaces $\F$, $\FE$ 
and $\Feps$ from \prettyref{prop:mingle-mangle}. 

\begin{rem}\label{rem:R_well-defined}
Let $(T^{E},T^{\K})$ be a strong, consistent family for $(\mathcal{F},E)$ and a generator for $(\mathcal{F}\nu,E)$.
Let $\F$ and $\FE$ be $\varepsilon$-into-compatible and the inclusion $\FV\hookrightarrow\F$ continuous. 
Consider a set of uniqueness $U$ for $(T^{\K},\mathcal{F}\nu)$ and a separating subspace $G\subset E'$.
For $u\in \F\varepsilon E$ such that $u(B_{\FV}^{\circ \F'})$ is bounded in $E$, 
i.e.\ $S(u)\in \Feps$, we set $f:=S(u)$. 
Then $f\in\FE$ by the $\varepsilon$-into-compatibility 
and we define $\widetilde{f}\colon U\to E$, $\widetilde{f}(x):=T^{E}(f)(x)$. This yields
\begin{equation}\label{eq:injective}
(e'\circ \widetilde{f})(x)= (e'\circ T^{E}(f))(x)=T^{\K}(e'\circ f)(x)
\end{equation}
for all $x\in U$ and $f_{e'}:=e'\circ f\in\F$ for each $e'\in E'$ by the strength of the family. 
Moreover, $T^{\K}_{x}(\cdot)\nu(x)\in B_{\FV}^{\circ \F'}$ for every $x\in\omega$, 
which implies that for every $e'\in E'$ there are $\alpha\in\mathfrak{A}$ and $C>0$ such that
\[
|f_{e'}|_{\FV}=\sup_{x\in\omega}\bigl|e'\bigl(u(T^{\K}_{x}(\cdot)\nu(x)\bigr)\bigr|
\leq C\sup_{y'\in B_{\FV}^{\circ \F'}}p_{\alpha}(u(y'))<\infty
\]
by strength and consistency. Hence $f_{e'}\in\FV$ for every $e'\in E'$ and $\widetilde{f}\in\mathcal{F}\nu_{G}(U,E)$.
\end{rem}

Under the assumptions of \prettyref{rem:R_well-defined} the map
\begin{equation}\label{eq:well_def_B_unique}
R_{U,G}\colon \Feps\to \mathcal{F}\nu_{G}(U,E),\;f\mapsto (T^{E}(f)(x))_{x\in U}, 
\end{equation}
is well-defined and linear. In addition, we derive from \eqref{eq:injective} that $R_{U,G}$ is injective since 
$U$ is a set of uniqueness and $G\subset E'$ separating. 

\begin{que}\label{que:surj_restr_set_unique}
Let the assumptions of \prettyref{rem:R_well-defined} be fulfilled.
When is the injective restriction map 
\[
R_{U,G}\colon \Feps\to \mathcal{F}\nu_{G}(U,E),\;f\mapsto (T^{E}(f)(x))_{x\in U},
\]
surjective?
\end{que}

Due to \prettyref{prop:mingle-mangle} c) the \prettyref{que:weak_strong} is a special case of this question 
if $\Lambda\subset\Omega=:\omega$ and $U:=\Lambda$ is a set of uniqueness for $(\id_{\K^{\Omega}},\mathcal{F}\nu)$.
To answer \prettyref{que:surj_restr_set_unique} for general sets of uniqueness we have to restrict to a certain class
of 'thick` separating subspaces of $E'$.

\begin{defn}[{determine  boundedness \cite[p.\ 230]{B/F/J}}]
A linear subspace $G\subset E'$ \emph{determines boundedness} if every $\sigma(E,G)$-bounded set $B\subset E$ is 
already bounded in $E$.
\end{defn}

$E'$ itself always determines boundedness by Mackey's theorem. Further examples can be found in 
\cite[3.10 Remark, p.\ 10]{kruse2018_3} and the references therein. 
We recall the following extension result for continuous linear operators. 

\begin{prop}[{\cite[Proposition 2.1, p.\ 691]{F/J/W}}]\label{prop:ext_B_set_uni}
Let $E$ be locally complete, $G\subset E'$ determine boundedness, $Z$ a Banach space whose closed unit ball $B_{Z}$ is a 
compact subset of an lcHs $Y$ and $X\subset Y'$ be a $\sigma(Y',Z)$-dense subspace. 
If $\mathsf{A}\colon X\to E$ is a $\sigma(X,Z)$-$\sigma(E,G)$-continuous linear map, 
then there exists a (unique) extension $\widehat{\mathsf{A}}\in Y\varepsilon E$ of $\mathsf{A}$ 
such that $\widehat{\mathsf{A}}(B_{Z}^{\circ Y'})$ is bounded in $E$ where 
$B_{Z}^{\circ Y'}:=\{y'\in Y'\;|\;\forall\;z\in B_{Z}:\; |y'(z)|\leq 1\}$.
\end{prop}

Now, we are able to generalise \cite[Theorem 2.2, p.\ 691]{F/J/W} and \cite[Theorem 10, p.\ 5]{jorda2013}.

\begin{thm}\label{thm:ext_B_unique}
Let $E$ be locally complete, $G\subset E'$ determine boundedness 
and $\F$ and $\FE$ be $\varepsilon$-into-compatible. 
Let $(T^{E},T^{\K})$ be a generator for $(\mathcal{F}\nu,E)$ 
and a strong, consistent family for $(\mathcal{F},E)$, $\FV$ a Banach space whose closed unit ball $B_{\FV}$ 
is a compact subset of $\F$ and $U$ a set of uniqueness for $(T^{\K},\mathcal{F}\nu)$.
Then the restriction map 
\[
 R_{U,G}\colon \Feps \to \mathcal{F}\nu_{G}(U,E) 
\]
is surjective.
\end{thm}
\begin{proof}
Let $f\in \mathcal{F}\nu_{G}(U,E)$. 
We set $X:=\operatorname{span}\{T^{\K}_{x}\;|\;x\in U\}$, $Y:=\F$ and $Z:=\FV$. 
The consistency of $(T^{E},T^{\K})$ for $(\mathcal{F},E)$ yields that $X\subset Y'$.
From $U$ being a set of uniqueness of $Z$ follows that $X$ is 
$\sigma(Z',Z)$-dense. Since $B_{Z}$ is a compact subset of $Y$, it follows that $Z$ is a linear subspace of $Y$ 
and the inclusion $Z\hookrightarrow Y$ is continuous, which yields $y'_{\mid Z}\in Z'$ for every $y'\in Y'$. 
Thus $X$ is $\sigma(Y',Z)$-dense. 
Let $\mathsf{A}\colon X\to E$ be the linear map determined by $\mathsf{A}(T^{\K}_{x}):=f(x)$. 
The map $\mathsf{A}$ is well-defined since $G$ is $\sigma(E',E)$-dense. 
Due to
\[
e'(\mathsf{A}(T^{\K}_{x}))=(e'\circ f)(x)=T^{\K}_{x}(f_{e'})
\]
for every $e'\in G$ and $x\in U$ we have that $\mathsf{A}$ is $\sigma(X,Z)$-$\sigma(E,G)$-continuous. 
We apply \prettyref{prop:ext_B_set_uni} and gain an extension $\widehat{\mathsf{A}}\in Y\varepsilon E$ 
of $\mathsf{A}$ such that $\widehat{\mathsf{A}}(B_{Z}^{\circ Y'})$ is bounded in $E$. 
We set $F:=S(\widehat{\mathsf{A}})\in\Feps$ and get for all $x\in U$ that
\[
T^{E}(F)(x)=T^{E}S(\widehat{\mathsf{A}})(x)=\widehat{\mathsf{A}}(T^{\K}_{x})=f(x)
\]
by consistency for $(\mathcal{F},E)$, implying $R_{U,G}(F)=f$.
\end{proof}

Let $\Omega\subset\R^{d}$ be open, $E$ an lcHs and 
$P(\partial)^{E}\colon\mathcal{C}^{\infty}(\Omega,E)\to\mathcal{C}^{\infty}(\Omega,E)$ 
a linear partial differential operator which is hypoelliptic if $E=\K$. We define the space
\[
\mathcal{C}^{\infty}_{P(\partial)}(\Omega,E):=\{f\in\mathcal{C}^{\infty}(\Omega,E)\;|\;f\in\ker P(\partial)^{E}\}
\]
of zero solutions and for a continuous weight $\nu\colon\Omega\to(0,\infty)$ the weighted space of zero solutions 
\[
  \mathcal{C}\nu^{\infty}_{P(\partial)}(\Omega,E)
:=\{f\in\mathcal{C}^{\infty}_{P(\partial)}(\Omega,E)\;|\;\forall\;\alpha\in\mathfrak{A}:\;
|f|_{\nu,\alpha}:=\sup_{x\in\Omega}p_{\alpha}(f(x))\nu(x)<\infty\}.
\]

\begin{prop}\label{prop:co_top_isomorphism}
Let $\Omega\subset\R^{d}$ be open, $E$ a locally complete lcHs and $P(\partial)^{\K}$ a hypoelliptic 
linear partial differential operator.
Then via $S_{(\mathcal{C}^{\infty}_{P(\partial)}(\Omega),\tau_{co})}$ 
holds $(\mathcal{C}^{\infty}_{P(\partial)}(\Omega),\tau_{co})\varepsilon E
\cong (\mathcal{C}^{\infty}_{P(\partial)}(\Omega,E),\tau_{co})$ 
and $(\mathcal{C}^{\infty}_{P(\partial)}(\Omega,E),\tau_{co})
=(\mathcal{C}^{\infty}_{P(\partial)}(\Omega,E),\tau_{\mathcal{C}^{\infty}})$ 
as locally convex spaces. 
\end{prop}
\begin{proof}
We already know that 
\[
S_{(\mathcal{C}^{\infty}_{P(\partial)}(\Omega),\tau_{\mathcal{C}^{\infty}})}\colon
(\mathcal{C}^{\infty}_{P(\partial)}(\Omega),\tau_{\mathcal{C}^{\infty}})\varepsilon E \to
(\mathcal{C}^{\infty}_{P(\partial)}(\Omega,E),\tau_{\mathcal{C}^{\infty}})
\]
is an isomorphism by \cite[Example 18 b), p.\ 1528]{kruse2017}. 
From $\tau_{co}=\tau_{\mathcal{C}^{\infty}}$ on $\mathcal{C}^{\infty}_{P(\partial)}(\Omega)$ by the hypoellipticity 
of $P(\partial)^{\K}$ (see e.g.\ \cite[p.\ 690]{F/J/W}) follows that 
$(\mathcal{C}^{\infty}_{P(\partial)}(\Omega),\tau_{co})\varepsilon E
=(\mathcal{C}^{\infty}_{P(\partial)}(\Omega),\tau_{\mathcal{C}^{\infty}})\varepsilon E$. 
Thus $S_{(\mathcal{C}^{\infty}_{P(\partial)}(\Omega),\tau_{co})}(u)
=S_{(\mathcal{C}^{\infty}_{P(\partial)}(\Omega),\tau_{\mathcal{C}^{\infty}})}(u)
\in\mathcal{C}^{\infty}_{P(\partial)}(\Omega,E)$ 
for all $u\in(\mathcal{C}^{\infty}_{P(\partial)}(\Omega),\tau_{co})\varepsilon E $. In particular, we obtain that 
\[
S_{(\mathcal{C}^{\infty}_{P(\partial)}(\Omega),\tau_{co})}\colon
(\mathcal{C}^{\infty}_{P(\partial)}(\Omega),\tau_{co})\varepsilon E \to
(\mathcal{C}^{\infty}_{P(\partial)}(\Omega,E),\tau_{\mathcal{C}^{\infty}})
\]
is an isomorphism. From the first part of the proof of \cite[Theorem 14, p.\ 1524]{kruse2017} with 
$(T^{E},T^{\K}):=(\id_{E^{\Omega}},\id_{\K^{\Omega}})$ we deduce that
\[
S_{(\mathcal{C}^{\infty}_{P(\partial)}(\Omega),\tau_{co})}\colon
(\mathcal{C}^{\infty}_{P(\partial)}(\Omega),\tau_{co})\varepsilon E \to
(\mathcal{C}^{\infty}_{P(\partial)}(\Omega,E),\tau_{co})
\]
is an isomorphism into and from
\[
S_{(\mathcal{C}^{\infty}_{P(\partial)}(\Omega),\tau_{co})}
\bigl((\mathcal{C}^{\infty}_{P(\partial)}(\Omega),\tau_{co})\varepsilon E\bigr)
=\mathcal{C}^{\infty}_{P(\partial)}(\Omega,E)
\]
that $(\mathcal{C}^{\infty}_{P(\partial)}(\Omega,E),\tau_{co})
=(\mathcal{C}^{\infty}_{P(\partial)}(\Omega,E),\tau_{\mathcal{C}^{\infty}})$ 
as locally convex spaces, which proves our statement. 
\end{proof}

\begin{prop}\label{prop:hypo_weighted_Banach}
If $\Omega\subset\R^{d}$ is open, $P(\partial)^{\K}$ a hypoelliptic linear partial differential operator and 
$\nu\colon\Omega\to(0,\infty)$ continuous,
then $\mathcal{C}\nu^{\infty}_{P(\partial)}(\Omega)$ is a Banach space.
\end{prop}
\begin{proof}
It suffices to prove that $\mathcal{C}\nu^{\infty}_{P(\partial)}(\Omega)$ is complete. 
Let $(f_{n})_{n\in\N}$ be a Cauchy sequence in $\mathcal{C}\nu^{\infty}_{P(\partial)}(\Omega)$. 
For every compact $K\subset\Omega$ we have
\begin{equation}\label{eq:hypo_weighted_Banach}
\sup_{x\in K}|f(x)|\leq \sup_{z\in K}\nu(z)^{-1}\sup_{x\in K}|f(x)|\nu(x)\leq \sup_{z\in K}\nu(z)^{-1}|f|_{\nu}, 
\quad f\in\mathcal{C}\nu^{\infty}_{P(\partial)}(\Omega),
\end{equation}
yielding that $(f_{n})_{n\in\N}$ is Cauchy sequence in $(\mathcal{C}^{\infty}_{P(\partial)}(\Omega),\tau_{co})$. 
$(\mathcal{C}^{\infty}_{P(\partial)}(\Omega),\tau_{co})$ is a Fr\'echet-Schwartz space 
(see e.g.\ \cite[p.\ 690]{F/J/W}), in particular complete, and thus $(f_{n})_{n\in\N}$ has a limit $f$ in 
$(\mathcal{C}^{\infty}_{P(\partial)}(\Omega),\tau_{co})$. 
Let $\varepsilon>0$ and $x\in\Omega$. Then there is $N_{\varepsilon,x}\in\N$ such that 
for all $m\geq N_{\varepsilon,x}$ it holds that
\[
|f_{m}(x)-f(x)|<\frac{\varepsilon}{2\nu(x)}.
\]
Further, there is $N_{\varepsilon}\in\N$ such that for all $n,m\geq N_{\varepsilon}$ it holds that
\[
|f_{n}-f_{m}|_{\nu}<\frac{\varepsilon}{2}.
\]
Hence for $n\geq N_{\varepsilon}$ we choose $m\geq\max(N_{\varepsilon},N_{\varepsilon,x})$ and derive 
\[
|f_{n}(x)-f(x)|\nu(x)\leq |f_{n}(x)-f_{m}(x)|\nu(x)+|f_{m}(x)-f(x)|\nu(x)
< \frac{\varepsilon}{2}+\frac{\varepsilon}{2\nu(x)}\nu(x)=\varepsilon.
\]
It follows that $|f_{n}-f|_{\nu}\leq \varepsilon$ and $|f|_{\nu}\leq \varepsilon+|f_{n}|_{\nu}$ 
for all $n\geq N_{\varepsilon}$, implying the convergence of $(f_{n})_{n\in\N}$ to $f$ 
in $\mathcal{C}\nu^{\infty}_{P(\partial)}(\Omega)$.
\end{proof}

\begin{cor}\label{cor:hypo_weighted_ext_unique}
Let $E$ be a locally complete lcHs, $G\subset E'$ determine boundedness, $\Omega\subset\R^{d}$ open, 
$P(\partial)^{\K}$ a hypoelliptic linear partial differential operator, $\nu\colon\Omega\to(0,\infty)$ continuous and 
$U$ a set of uniqueness for $(\id_{\K^{\Omega}},\mathcal{C}\nu^{\infty}_{P(\partial)})$.
If $f\colon U\to E$ is a function such that $e'\circ f$ admits an extension 
$f_{e'}\in\mathcal{C}\nu^{\infty}_{P(\partial)}(\Omega)$ 
for every $e'\in G$, then there exists a unique extension $F\in\mathcal{C}\nu^{\infty}_{P(\partial)}(\Omega,E)$ of $f$.
\end{cor}
\begin{proof}
We choose $\F:=(\mathcal{C}^{\infty}_{P(\partial)}(\Omega),\tau_{co})$ 
and $\FE:=(\mathcal{C}^{\infty}_{P(\partial)}(\Omega,E),\tau_{co})$. Then we  
have $\mathcal{F}\nu(\Omega)=\mathcal{C}\nu^{\infty}_{P(\partial)}(\Omega)$ and 
$\mathcal{F}\nu(\Omega,E)=\mathcal{C}\nu^{\infty}_{P(\partial)}(\Omega,E)$
with the generator $(T^{E},T^{\K}):=(\operatorname{\id}_{E^{\Omega}},\operatorname{\id}_{\K^{\Omega}})$ 
for $(\mathcal{F}\nu,E)$.
We note that $\F$ and $\FE$ are $\varepsilon$-compatible 
and $(T^{E},T^{\K})$ is a strong, consistent family for $(\mathcal{F},E)$ by \prettyref{prop:co_top_isomorphism}. 
We observe that $\mathcal{F}\nu(\Omega)$ is a Banach space by \prettyref{prop:hypo_weighted_Banach} and 
for every compact $K\subset\Omega$ we have 
\[
\sup_{x\in K}|f(x)|\underset{\eqref{eq:hypo_weighted_Banach}}{\leq}\sup_{z\in K}\nu(z)^{-1}|f|_{\nu} 
\leq \sup_{z\in K}\nu(z)^{-1},
\quad f\in B_{\mathcal{F}\nu(\Omega)},
\]
yielding that $B_{\mathcal{F}\nu(\Omega)}$ is bounded in $\F$.  
The space $\F=(\mathcal{C}^{\infty}_{P(\partial)}(\Omega),\tau_{co})$ is a Fr\'echet-Schwartz space, 
thus a Montel space, and it is easy to check that $B_{\mathcal{F}\nu(\Omega)}$ is $\tau_{co}$-closed. 
Hence the bounded and $\tau_{co}$-closed set $B_{\mathcal{F}\nu(\Omega)}$ is compact in $\F$. 
Finally, we remark that the $\varepsilon$-compatibility of $\F$ and $\FE$ 
in combination with the consistency of $(\id_{E^{\Omega}},\id_{\K^{\Omega}})$ for $(\mathcal{F},E)$ gives
$\Feps=\FVE$ as linear spaces by \prettyref{prop:mingle-mangle} c). 
From \prettyref{thm:ext_B_unique} follows our statement.
\end{proof}

If $\Omega=\D\subset\C$ is the open unit disc, $P(\partial)=\overline{\partial}$ the Cauchy--Riemann operator 
and $\nu=1$ on $\D$, then $\mathcal{C}\nu^{\infty}_{P(\partial)}(\Omega,E)=H^{\infty}(\D,E)$ and 
a sequence $U:=(z_{n})_{n\in\N}\subset\D$ of distinct elements is a set of uniqueness for $(\id_{\C^{\D}},H^{\infty})$ 
if and only if it satisfies the Blaschke condition $\sum_{n\in\N}(1-|z_{n}|)=\infty$ 
(see e.g.\ \cite[15.23 Theorem, p.\ 303]{rudin1970}).

For a continuous function $\nu\colon\D\to(0,\infty)$ and a complex lcHs $E$ we define the Bloch type spaces 
\[
\mathcal{B}\nu(\D,E):=\{f\in\mathcal{O}(\D,E)\;|\;\forall\;\alpha\in\mathfrak{A}:\;|f|_{\nu,\alpha}<\infty\}
\]
with 
\[
|f|_{\nu,\alpha}:=\max\bigl(p_{\alpha}(f(0)),\sup_{z\in\D}p_{\alpha}((\partial_{\C}^{1})^{E}f(z))\nu(z)\bigr)
\]
and the complex derivative
\[
(\partial_{\C}^{1})^{E}f(z):=\lim_{\substack{h\to 0\\ h\in\C,h\neq 0}}\frac{f(z+h)-f(z)}{h},
\quad z\in\D,\;f\in\mathcal{O}(\D,E).
\]
If $E=\C$, we write $f'(z):=(\partial_{\C}^{1})^{\C}f(z)$ for $z\in\D$ and $f\in\mathcal{O}(\D)$. 

\begin{prop}\label{prop:Bloch_Banach}
If $\nu\colon\D\to(0,\infty)$ is continuous, then $\mathcal{B}\nu(\D)$ is a Banach space.
\end{prop}
\begin{proof}
Let $f\in\mathcal{B}\nu(\D)$. From the estimates 
\begin{align*}
|f(z)|&\leq |f(0)|+\bigl|\int_{0}^{z}f'(\zeta)\d\zeta\bigr|
\leq |f(0)|+\frac{|z|}{\min_{\xi\in[0,z]}\nu(\xi)}\sup_{\zeta\in[0,z]}|f'(\zeta)|\nu(\zeta)\\
&\leq 2\max\Bigl(1,\frac{|z|}{\min_{\xi\in[0,z]}\nu(\xi)}\Bigr)|f|_{\nu}
\end{align*}
for every $z\in\D$ and 
\begin{equation}\label{eq:Bloch}
\max_{|z|\leq r}|f(z)|\leq  2\max\Bigl(1,\frac{r}{\min_{|z|\leq r}\nu(z)}\Bigr)|f|_{\nu}
\end{equation}
for all $0<r<1$ and $f\in\mathcal{B}\nu(\D)$ it follows that 
$\mathcal{B}\nu(\D)$ is a Banach space 
by using the completeness of $(\mathcal{O}(\D),\tau_{co})$ analogously to the proof 
of \prettyref{prop:hypo_weighted_Banach}.
\end{proof}

Let $E$ be an lcHs and $\nu\colon\D\to(0,\infty)$ be continuous. 
We set $\omega:=\{0\}\cup\{(1,z)\;|\;z\in\D\}$, 
define the operator $T^{E}\colon \mathcal{O}(\D,E)\to E^{\omega}$ by 
\[
T^{E}(f)(0):=f(0)\quad\text{and}\quad T^{E}(f)(1,z):=(\partial_{\C}^{1})^{E}f(z),\;z\in\D,
\] 
and the weight $\nu_{\ast}\colon\omega\to (0,\infty)$ by 
\[
\nu_{\ast}(0):=1\quad\text{and}\quad\nu_{\ast}(1,z):=\nu(z),\;z\in\D.
\]
Then we have for every $\alpha\in\mathfrak{A}$ that
\[
 |f|_{\nu,\alpha}=\sup_{x\in\omega}p_{\alpha}\bigl(T^{E}(f)(x)\bigr)\nu_{\ast}(x),
 \quad f\in\mathcal{B}\nu(\D,E),
\]
and with $\mathcal{F}(\D,E):=\mathcal{O}(\D,E)$ we observe that 
$\mathcal{F}\nu_{\ast}(\D,E)=\mathcal{B}\nu(\D,E)$ with generator $(T^{E},T^{\C})$.

\begin{cor}\label{cor:Bloch_ext_unique}
Let $E$ be a locally complete lcHs, $G\subset E'$ determine boundedness, $\nu\colon\D\to(0,\infty)$ continuous 
and $U_{\ast}\subset\D$ have an accumulation point in $\D$. 
If $f\colon \{0\}\cup(\{1\}\times U_{\ast})\to E$ is a function such that there is 
$f_{e'}\in\mathcal{B}\nu(\D)$ for each $e'\in G$ with $f_{e'}(0)=e'(f(0))$ and 
$f_{e'}'(z)=e'(f(1,z))$ for all $z\in U_{\ast}$, 
then there exists a unique $F\in\mathcal{B}\nu(\D,E)$ with $F(0)=f(0)$ and $(\partial_{\C}^{1})^{E}F(z)=f(1,z)$ 
for all $z\in U_{\ast}$.
\end{cor}
\begin{proof}
We take $\mathcal{F}(\D):=(\mathcal{O}(\D),\tau_{co})$ 
and $\mathcal{F}(\D,E):=(\mathcal{O}(\D,E),\tau_{co})$. Then 
we have $\mathcal{F}\nu_{\ast}(\D)=\mathcal{B}\nu(\D,E)$ 
and $\mathcal{F}\nu_{\ast}(\Omega,E)=\mathcal{B}\nu(\D,E)$ 
with the weight $\nu_{\ast}$ and generator $(T^{E},T^{\C})$
for $(\mathcal{F}\nu_{\ast},E)$ described above. 
The spaces $\mathcal{F}(\D)$ and $\mathcal{F}(\D,E)$ are $\varepsilon$-compatible by 
\prettyref{prop:co_top_isomorphism} and the generator is a strong, consistent family for $(\mathcal{F},E)$ 
(see e.g.\ \cite[Theorem 4.5, p.\ 368]{kruse2018_1}). 
Due to \prettyref{prop:Bloch_Banach} $\mathcal{F}\nu_{\ast}(\D)=\mathcal{B}\nu(\D)$ is a Banach space
and we deduce from \eqref{eq:Bloch} that $B_{\mathcal{F}\nu_{\ast}(\D)}$ is compact 
in the Montel space $(\mathcal{O}(\D),\tau_{co})$.
We note that the $\varepsilon$-compatibility of $\F$ and $\FE$ 
in combination with the consistency of $(T^{E},T^{\C})$ for $(\mathcal{F},E)$ gives
$\mathcal{F}_{\varepsilon}\nu_{\ast}(\D,E)=\mathcal{F}\nu_{\ast}(\D,E)$ as linear spaces by 
\prettyref{prop:mingle-mangle} c). In addition, $U:=\{0\}\cup\{(1,z)\;|\;z\in U_{\ast}\}$ is a set of uniqueness 
for $(T^{\C},\mathcal{F}\nu_{\ast})$ by the identity theorem, proving our statement by \prettyref{thm:ext_B_unique}.
\end{proof}
\section{Extension of locally bounded functions}
In order to obtain an affirmative answer to \prettyref{que:surj_restr_set_unique} 
for general separating subspaces of $E'$ we have to restrict to a certain class of `thick' sets of uniqueness.

\begin{defn}[{fix the topology}]\label{def:fix_top_1}
Let $\FV$ be a Hausdorff space. $U\subset\omega$ \emph{fixes the topology} in $\FV$ if there is $C>0$ such that 
\[
|f|_{\FV}\leq C \sup_{x\in U}|T^{\K}(f)(x)|\nu(x),\quad f\in \FV .
\]
\end{defn}

In particular, $U$ is a set of uniqueness if it fixes the topology. The present definition of fixing 
the topology is a special case of \cite[4.1 Definition, p.\ 18]{kruse2018_3}. 
Sets that fix the topolgy appear under many different names, e.g.\ dominating, (weakly) sufficient, sampling sets 
(see \cite[p.\ 18--19]{kruse2018_3} and the references therein), and they are related to $\ell\nu(U)$-frames 
used by Bonet et.\ al in \cite{bonet2017}. 
For a set $U$, a function $\nu\colon U\to(0,\infty)$ and an lcHs $E$ we set 
\begin{equation}\label{eq:frame}
\ell\nu(U,E):=\{f\colon U\to E\;|\;\forall\;\alpha\in\mathfrak{A}:\;
\|f\|_{\alpha}:=\sup_{x\in U}p_{\alpha}(f(x))\nu(x)<\infty\}.
\end{equation}
If $U$ is countable and fixes the topology in $\FV$, the inclusion $\ell\nu(U)\hookrightarrow (\K^{U},\tau_{co})$ 
is continuous and $\ell\nu(U)$ contains the space of sequences (on $U$) with compact support as a linear subspace, 
then $(T^{\K}_{x})_{x\in U}$ is an $\ell\nu(U)$-frame in the sense of \cite[Definition 2.1, p.\ 3]{bonet2017}. 
The next definition is a special case of \cite[4.2 Definition, p.\ 19]{kruse2018_3}.

\begin{defn}[{$lb$-restriction space}]
Let $\FV$ be a Hausdorff space, $U$ fix the topology in $\FV$ and $G\subset E'$ a separating subspace. We set 
\[
N_{U}(f):=\{f(x)\nu(x)\;|\;x\in U\}
\]
for $f\in\mathcal{FV}_{G}(U,E)$ and 
\begin{align*}
\mathcal{FV}_{G}(U,E)_{lb}:=&\{f\in\mathcal{FV}_{G}(U,E)\;|\;N_{U}(f)\;\text{bounded in}\; E\}\\
=&\mathcal{FV}_{G}(U,E)\cap\ell\nu(U,E).
\end{align*}
\end{defn}

Let us recall the assumptions of \prettyref{rem:R_well-defined} but now $U$ fixes the topology.
Let $(T^{E},T^{\K})$ be a strong, consistent family for $(\mathcal{F},E)$ 
and a generator for $(\mathcal{F}\nu,E)$.
Let $\F$ and $\FE$ be $\varepsilon$-into-compatible and the inclusion $\FV\hookrightarrow\F$ continuous. 
Consider a set $U$ which fixes the topology in $\FV$ 
and a separating subspace $G\subset E'$.
For $u\in \F\varepsilon E$ such that $u(B_{\FV}^{\circ \F'})$ is bounded in $E$
we have $R_{U,G}(f)\in\mathcal{F}\nu_{G}(U,E)$ with $f:=S(u)\in\Feps$ by \eqref{eq:well_def_B_unique}.
Further, $T^{\K}_{x}(\cdot)\nu(x)\in B_{\FV}^{\circ \F'}$ for every $x\in\omega$, 
which implies that 
\[
\sup_{x\in U}p_{\alpha}(R_{U,G}(f)(x))\nu(x)
=\sup_{x\in U}p_{\alpha}\bigl(u(T^{\K}_{x}(\cdot)\nu(x))\bigr)\\
\leq \sup_{y'\in B_{\FV}^{\circ \F'}}p_{\alpha}(u(y'))<\infty
\]
for all $\alpha\in\mathfrak{A}$ by consistency. Hence $R_{U,G}(f)\in\mathcal{F}\nu_{G}(U,E)_{lb}$.
Therefore the injective linear map
\[
R_{U,G}\colon \Feps \to \mathcal{F}\nu_{G}(U,E)_{lb},\;f\mapsto (T^{E}(f)(x))_{x\in U},
\]
is well-defined and the question we want to answer is:

\begin{que}
Let the assumptions of \prettyref{rem:R_well-defined} be fulfilled and $U$ fix the topology in $\FV$.
When is the injective restriction map 
\[
R_{U,G}\colon \Feps\to \mathcal{FV}_{G}(U,E)_{lb},\;f\mapsto (T^{E}(f)(x))_{x\in U},
\]
surjective?
\end{que}

\begin{prop}[{\cite[Proposition 3.1, p.\ 692]{F/J/W}}]\label{prop:ext_B_fix_top}
Let $E$ be locally complete, $G\subset E'$ a separating subspace and $Z$ a Banach space whose closed unit ball $B_{Z}$ is a 
compact subset of an lcHs $Y$. Let $B_{1}\subset B_{Z}^{\circ Y'}$ such that 
$B_{1}^{\circ Z}:=\{z\in Z\;|\;\forall\;y'\in B_{1}:\;|y'(z)|\leq 1\}$ is bounded in $Z$.  
If $\mathsf{A}\colon X:=\operatorname{span}B_{1}\to E$ is a linear map which is bounded on $B_{1}$ 
such that there is a $\sigma(E',E)$-dense subspace $G\subset E'$ with $e'\circ \mathsf{A}\in Z$ for all $e'\in G$, 
then there exists a (unique) extension $\widehat{\mathsf{A}}\in Y\varepsilon E$ of $\mathsf{A}$ 
such that $\widehat{\mathsf{A}}(B_{Z}^{\circ Y'})$ is bounded in $E$.
\end{prop}

The following theorem is a generalisation of \cite[Theorem 3.2, p.\ 693]{F/J/W} and \cite[Theorem 12, p.\ 5]{jorda2013}.

\begin{thm}\label{thm:ext_B_fix_top}
Let $E$ be locally complete, $G\subset E'$ a separating subspace and $\F$ 
and $\FE$ be $\varepsilon$-into-compatible. 
Let $(T^{E},T^{\K})$ be a generator for $(\mathcal{F}\nu,E)$ 
and a strong, consistent family for $(\mathcal{F},E)$, $\FV$ a Banach space
whose closed unit ball $B_{\FV}$ 
is a compact subset of $\F$ and $U$ fix the topology in $\FV$.
Then the restriction map 
\[
 R_{U,G}\colon \Feps \to \mathcal{F}\nu_{G}(U,E)_{lb} 
\]
is surjective.
\end{thm}
\begin{proof}
Let $f\in \mathcal{F}\nu_{G}(U,E)_{lb}$. 
We set $B_{1}:=\{T^{\K}_{x}(\cdot)\nu(x)\;|\;x\in U\}$, 
$X:=\operatorname{span}B_{1}$, $Y:=\F$ and $Z:=\FV$. 
We have $B_{1}\subset Y'$ since $(T^{E},T^{\K})$ is a consistent family for $(\mathcal{F},E)$. 
If $f\in B_{Z}$, then
\[
|T^{\K}_{x}(f)\nu(x)|\leq |f|_{\FV}\leq 1
\]
for all $x\in U$ and thus $B_{1}\subset B_{Z}^{\circ Y'}$. 
Further on, there is $C>0$ such that for all $f\in B_{1}^{\circ Z}$
\[
|f|_{\FV}\leq C\sup_{x\in U}|T^{\K}_{x}(f)|\nu(x)
\leq C
\]
as $U$ fixes the topology in $Z$, implying the boundedness of $B_{1}^{\circ Z}$ in $Z$.
Let $\mathsf{A}\colon X\to E$ be the linear map determined by 
\[
\mathsf{A}(T^{\K}_{x}(\cdot)\nu(x)):=f(x)\nu(x). 
\]
The map $\mathsf{A}$ is well-defined since $G$ is $\sigma(E',E)$-dense, and 
bounded on $B_{1}$ because $\mathsf{A}(B_{1})=N_{U}(f)$.
Let $e'\in G$ and $f_{e'}$ be the unique element in $\mathcal{F}\nu(\Omega)$ such that 
$T^{\K}(f_{e'})(x)=(e'\circ f)(x)$ for all $x \in U$, which implies 
$T^{\K}(f_{e'})(x)\nu(x)=(e'\circ \mathsf{A})(T^{\K}_{x}(\cdot)\nu(x))$.
Again, this equation allows us to consider $f_{e'}$ as a linear form on $X$ 
(by setting $f_{e'}(T^{\K}_{x}(\cdot)\nu(x)):=(e'\circ \mathsf{A})(T^{\K}_{x}(\cdot)\nu(x))$), 
which yields $e'\circ \mathsf{A}\in\mathcal{F}\nu(\Omega)=Z$ for all $e'\in G$. 
Hence we can apply \prettyref{prop:ext_B_fix_top} and obtain an extension 
$\widehat{\mathsf{A}}\in Y\varepsilon E$ of $\mathsf{A}$ 
such that $\widehat{\mathsf{A}}(B_{Z}^{\circ Y'})$ is bounded in $E$. 
We set $F:=S(\widehat{\mathsf{A}})\in\Feps$ and get for all $x\in U$ that
\[
T^{E}(F)(x)=T^{E}S(\widehat{\mathsf{A}})(x)=\widehat{\mathsf{A}}(T^{\K}_{x})
=\frac{1}{\nu(x)}\mathsf{A}(T^{\K}_{x}(\cdot)\nu(x))=f(x)
\]
by consistency for $(\mathcal{F},E)$, yielding $R_{U,G}(F)=f$.
\end{proof}

\begin{cor}\label{cor:hypo_weighted_ext_fix_top}
Let $E$ be a locally complete lcHs, $G\subset E'$ a separating subspace, $\Omega\subset\R^{d}$ open, 
$P(\partial)^{\K}$ a hypoelliptic linear partial differential operator, $\nu\colon\Omega\to(0,\infty)$ continuous and 
$U$ fix the topology in $\mathcal{C}\nu^{\infty}_{P(\partial)}(\Omega)$.
If $f\colon U\to E$ is a function in $\ell\nu(U,E)$ such that $e'\circ f$ 
admits an extension $f_{e'}\in\mathcal{C}\nu^{\infty}_{P(\partial)}(\Omega)$ 
for every $e'\in G$, then there exists a unique extension $F\in\mathcal{C}\nu^{\infty}_{P(\partial)}(\Omega,E)$ of $f$.
\end{cor}
\begin{proof}
Observing that $f\in\mathcal{F}\nu_{G}(U,E)_{lb}$ with $\mathcal{F}\nu(\Omega)
=\mathcal{C}\nu^{\infty}_{P(\partial)}(\Omega)$, 
our statement follows directly from \prettyref{thm:ext_B_fix_top} 
whose conditions are fulfilled by the proof of \prettyref{cor:hypo_weighted_ext_unique}. 
\end{proof}

Sets that fix the topology in $\mathcal{C}\nu^{\infty}_{P(\partial)}(\Omega)$ for different weights $\nu$ 
are well-studied if $P(\partial)=\overline{\partial}$ is the Cauchy--Riemann operator.
If $\Omega\subset\C$ is open, $P(\partial)=\overline{\partial}$ and $\nu=1$, 
then $\mathcal{C}\nu^{\infty}_{P(\partial)}(\Omega)=H^{\infty}(\Omega)$ is the space of bounded holomorphic functions 
on $\Omega$. 
Brown, Shields and Zeller characterise the countable discrete sets $U:=(z_{n})_{n\in\N}\subset\Omega$ that 
fix the topology in $H^{\infty}(\Omega)$ 
with $C=1$ and equality in \prettyref{def:fix_top_1} for Jordan domains $\Omega$ in 
\cite[Theorem 3, p.\ 167]{brownshieldszeller1960}.
In particular, they prove for $\Omega=\D$ that a discrete $U=(z_{n})_{n\in\N}$ fixes the topology in 
$H^{\infty}(\D)$ if and only if almost 
every boundary point is a non-tangential limit of a sequence in $U$. 
Bonsall obtains the same characterisation for bounded harmonic functions, i.e.\ $P(\partial)=\Delta$ and 
$\nu=1$, on $\Omega=\D$ by \cite[Theorem 2, p.\ 473]{bonsall1987}. 
An example of such a set $U=(z_{n})_{n\in\N}\subset\D$ is constructed in
\cite[Remark 6, p.\ 172]{brownshieldszeller1960}.
Probably the first example of a countable discrete set $U\subset\D$ that fixes the topology in $H^{\infty}(\D)$ 
is given by Wolff in \cite[p.\ 1327]{wolff1921} (cf.\ \cite[Theorem (Wolff), p.\ 402]{grosse-erdmann2004}). 
In \cite[4.14 Theorem, p.\ 255]{rubelshields1966} Rubel and Shields give a charaterisation of sets $U\subset\Omega$ 
that fix the topology in $H^{\infty}(\Omega)$ by means of bounded complex measures where $\Omega\subset\C$ is open 
and connected. The existence of a countable $U$ fixing the topology in $H^{\infty}(\Omega)$ is shown in 
\cite[4.15 Proposition, p.\ 256]{rubelshields1966}.
In the case of several complex variables the existence of such a countable $U$ is treated by Sibony in 
\cite[Remarques 4 b), p.\ 209]{sibony1975}
and by Massaneda and Thomas in \cite[Theorem 2, p.\ 838]{massaneda2000}.

If $\Omega=\C$ and $P(\partial)=\overline{\partial}$, 
then $\mathcal{C}\nu^{\infty}_{P(\partial)}(\Omega)=:F_{\nu}^{\infty}(\C)$ is 
a generalised $L^{\infty}$-version of the Bargmann--Fock space. 
In the case that $\nu(z)=\exp(-\alpha |z|^{2}/2)$, $z\in\C$, for some $\alpha>0$, Seip and Wallst\'en show 
in \cite[Theorem 2.3, p.\ 93]{seip1992b} that a countable discrete set $U\subset\C$ fixes the topology 
in $F_{\nu}^{\infty}(\C)$ if and only if $U$ contains a uniformly discrete subset $U'$ with lower uniform 
density $D^{-}(U')>\alpha/\pi$ (the proof of sufficiency is given in \cite{seip1992c} and the result was announced 
in \cite[Theorem 1.3, p.\ 324]{seip1992a}). 
A generalisation of this result using lower angular densities is given by Lyubarski{\u{\i}} and Seip 
in \cite[Theorem 2.2, p.\ 162]{lyubarskij1994} to weights of the form $\nu(z)=\exp(-\phi(\arg z)|z|^{2}/2)$, $z\in\C$,
with a $2\pi$-periodic $2$-trigonometrically convex function $\phi$ such that $\phi\in\mathcal{C}^{2}([0,2\pi])$ 
and $\phi(\theta)+(1/4)\phi''(\theta)>0$ for all $\theta\in[0,2\pi]$.
An extension of the results in \cite{seip1992b} to weights of the form $\nu(z)=\exp(-\phi(z))$, $z\in\C$, 
with a subharmonic function $\phi$ such that $\Delta\phi(z)\sim 1$ is given in 
\cite[Theorem 1, p.\ 249]{ortegacerda1998} by Ortega-Cerd{\`a} and Seip. 
Here, $f(x)\sim g(x)$ for two functions $f,g\colon\Omega\to\R$ means 
that there are $C_{1},C_{2}>0$ such that $C_{1}g(x)\leq f(x)\leq C_{2}g(x)$ for all $x\in\Omega$.  
Marco, Massaneda and Ortega-Cerd{\`a} describe sets that fix the topology in $F_{\nu}^{\infty}(\C)$ with 
$\nu(z)=\exp(-\phi(z))$, $z\in\C$, for some subharmonic function $\phi$ whose Laplacian $\Delta\phi$ is a 
doubling measure (see \cite[Definition 5, p.\ 868]{marco2003}), 
e.g.\ $\phi(z)=|z|^{\beta}$ for some $\beta>0$ in \cite[Theorem A, p.\ 865]{marco2003}.
The case of several complex variables is handled by Ortega-Cerd{\`a}, Schuster and Varolin 
in \cite[Theorem 2, p.\ 81]{ortegacerda2006}.

If $\Omega=\D$ and $P(\partial)=\overline{\partial}$, then 
$\mathcal{C}\nu^{\infty}_{P(\partial)}(\Omega)=:A_{\nu}^{\infty}(\D)$ is 
a generalised $L^{\infty}$-version of the weighted Bergman space (and of $H^{\infty}(\D)$). 
For $\nu(z)=(1-|z|^{2})^{n}$, $z\in\D$, for some $n\in\N$, Seip proves that a countable discrete set 
$U\subset\D$ fixes the topology in $A_{\nu}^{\infty}(\D)$ if and only if $U$ contains a uniformly discrete subset 
$U'$ with lower uniform density $D^{-}(U')>n$ by \cite[Theorem 1.1, p.\ 23]{seip1993}, and gives a typical example 
in \cite[p.\ 23]{seip1993}.
Later on, this is extended by Seip in \cite[Theorem 2, p.\ 718]{seip1998} to weights $\nu(z)=\exp(-\phi(z))$, $z\in\D$, 
with a subharmonic function $\phi$ such that $\Delta\phi(z)\sim(1-|z|^{2})^{-2}$, 
e.g.\ $\phi(z)=-\beta\ln(1-|z|^2)$, $z\in\D$, for some $\beta>0$.
Doma\'nski and Lindstr\"{o}m give necessary resp.\ sufficient conditions for fixing the topology 
in $A_{\nu}^{\infty}(\D)$ in the case that $\nu$ is an essential weight on $\D$, i.e.\ there is $C>0$ with
$\nu(z)\leq\widetilde{\nu}(z)\leq C\nu(z)$ for each $z\in\D$ where 
$\widetilde{\nu}(z):=(\sup\{|f(z)|\;|\;f\in B_{A_{\nu}^{\infty}(\D)}\})^{-1}$ is the associated weight. 
In \cite[Theorem 29, p.\ 260]{DomLind2002} they describe necessary resp.\ sufficient conditions for fixing the topology 
if the upper index $U_{\nu}$ is finite (see \cite[p.\ 242]{DomLind2002}), and 
necessary and sufficient conditions in \cite[Corollary 31, p.\ 261]{DomLind2002} 
if $0<L_{\nu}=U_{\nu}<\infty$ holds where $L_{\nu}$ is the lower index (see \cite[p.\ 243]{DomLind2002}), 
which for example can be applied to $\nu(z)=(1-|z|^{2})^{n}(\ln(\tfrac{e}{1-|z|}))^{\beta}$, $z\in\D$, 
for some $n>0$ and $\beta\in\R$. 
The case of simply connected open $\Omega\subset\C$ is considered in \cite[Corollary 32, p.\ 261--262]{DomLind2002}.

Borichev, Dhuez and Kellay treat $A_{\nu}^{\infty}(\D)$ and $F_{\nu}^{\infty}(\C)$ simultaneously.  
Let $\Omega_{R}:=\D$, if $R=1$, and $\Omega_{R}:=\C$ if $R=\infty$. They take $\nu(z)=\exp(-\phi(z))$, $z\in\Omega_{R}$, 
where $\phi\colon[0,R)\to[0,\infty)$ is an increasing function such that $\phi(0) = 0$, 
$\lim_{r\to R}\phi(r)=\infty$, $\phi$ is extended to $\Omega_{R}$ by $\phi(z):=\phi(|z|)$, 
$\phi\in\mathcal{C}^{2}(\Omega_{R})$, and, in addition $\Delta\phi(z)\geq 1$ if $R=\infty$ 
(see \cite[p.\ 564--565]{borichev2007}). 
Then they set $\rho\colon[0,R)\to\R$, $\rho(r):=[\Delta\phi(r)]^{-1/2}$, and suppose that $\rho$
decreases to $0$ near $R$, $\rho'(r)\to 0$, $r\to R$, and either
$(I_{\D})$ the function $r\mapsto\rho(r)(1-r)^{-C}$ increases for some $C\in\R$ and for $r$ close to $1$, resp.\
$(I_{\C})$ the function $r\mapsto\rho(r)r^{C}$ increases for some $C\in\R$ and for large $r$,
or $(II_{\Omega_{R}})$ that $\rho'(r)\ln(1/\rho(r))\to 0$, $r\to R$ (see \cite[p.\ 567--569]{borichev2007}).
Typical examples for $(I_{\D})$ are 
\[
\phi(r)=\ln(\ln(\tfrac{1}{1-r}))\ln(\tfrac{1}{1-r})\quad\text{or}\quad\phi(r)=\tfrac{1}{1-r},
\]
a typical example for $(II_{\D})$ is $\phi(r)=\exp(\tfrac{1}{1-r})$, for $(I_{\C})$
\[
 \phi(r)=r^2\ln(\ln(r))\quad\text{or}\quad\phi(r)=r^p,\;\text{for some}\; p>2,
\]
and a typical example for $(II_{\C})$ is $\phi(r)=\exp(r)$.
Sets that fix the topology in $A_{\nu}^{\infty}(\D)$ are described by densities 
in \cite[Theorem 2.1, p.\ 568]{borichev2007} and sets that fix the topology in $F_{\nu}^{\infty}(\C)$ 
in \cite[Theorem 2.5, p.\ 569]{borichev2007}.

Wolf uses sets that fix the topology in $A_{\nu}^{\infty}(\D)$ for the characterisation of weighted composition 
operators on $A_{\nu}^{\infty}(\D)$ with closed range in \cite[Theorem 1, p.\ 36]{wolf2011} for bounded $\nu$.

\begin{cor}\label{cor:Bloch_ext_fix_top}
Let $E$ be a locally complete lcHs, $G\subset E'$ a separating subspace, $\nu\colon\D\to(0,\infty)$ continuous and 
$U:=\{0\}\cup(\{1\}\times U_{\ast})$ fix the topology in $\mathcal{B}\nu(\D)$ with $U_{\ast}\subset\D$. 
If $f\colon U\to E$ is a function in $\ell\nu_{\ast}(U,E)$ such that there is 
$f_{e'}\in\mathcal{B}\nu(\D)$ for each $e'\in G$ with $f_{e'}(0)=e'(f(0))$ and $f_{e'}'(z)=e'(f(1,z))$ 
for all $z\in U_{\ast}$, then there exists a unique $F\in\mathcal{B}\nu(\D,E)$ with $F(0)=f(0)$ and 
$(\partial_{\C}^{1})^{E}F(z)=f(1,z)$ for all $z\in U_{\ast}$.
\end{cor}
\begin{proof}
As in \prettyref{cor:hypo_weighted_ext_fix_top} but with $\mathcal{F}\nu_{\ast}(\D)=\mathcal{B}\nu(\D)$ 
and \prettyref{cor:Bloch_ext_unique} instead of \prettyref{cor:hypo_weighted_ext_unique}.
\end{proof}

Sets that fix the topology in $\mathcal{B}\nu(\D)$ play an important role in the characterisation of 
composition operators on $\mathcal{B}\nu(\D)$ with closed range. 
Chen and Gauthier give a characterisation in \cite{chengauthier2008} for weights of the form 
$\nu(z)=(1-|z|^{2})^{\alpha}$, $z\in\D$, for some $\alpha\geq 1$. 
We recall the following definitions which are needed to phrase this characterisation.
For a continuous function $\nu\colon\D\to(0,\infty)$ and a non-constant holomorphic function $\phi\colon\D\to\D$ we set 
\[
\tau^{\nu}_{\phi}(z):=\frac{\nu(z)|\phi'(z)|}{\nu(\phi(z))},\;z\in\D,\quad\text{and}\quad
\Omega^{\nu}_{\varepsilon}:=\{z\in\D\;|\;\tau^{\nu}_{\phi}(z)\geq\varepsilon\},\;\varepsilon>0,
\]
and define the \emph{pseudohyperbolic distance}
\[
\rho(z,w):=\Bigl|\frac{z-w}{1-\overline{z}w}\Bigr|,\;z,w\in\D,
\]
(see \cite[p.\ 195-196]{chengauthier2008}). For $0<r<1$ a set $E\subset\D$ is called a \emph{pseudo} $r$\emph{-net} 
if for every $w\in\D$ there is $z\in\D$ with $\rho(z,w)\leq r$ (see \cite[p.\ 198]{chengauthier2008}). 

\begin{thm}[{\cite[Theorem 3.1, p.\ 199, Theorem 4.3, p.\ 202]{chengauthier2008}}]\label{thm:Bloch_sampling}
Let $\phi\colon\D\to\D$ be a non-constant holomorphic function and $\nu(z)=(1-|z|^{2})^{\alpha}$, $z\in\D$, 
for some $\alpha\geq 1$. Then the following statements are equivalent.
\begin{enumerate}
\item[(i)] The composition operator $C_{\phi}\colon\mathcal{B}\nu(\D)\to\mathcal{B}\nu(\D)$, $C_{\phi}(f):=f\circ\phi$, is bounded below 
(i.e.\ has closed range).
\item[(ii)] There is $\varepsilon>0$ such that $\{0\}\cup(\{1\}\times \phi(\Omega^{\nu}_{\varepsilon}))$ fixes the topology in $\mathcal{B}\nu(\D)$.
\item[(iii)] There are $\varepsilon>0$ and $0<r<1$ such that $\phi(\Omega^{\nu}_{\varepsilon})$ is a pseudo $r$-net.
\end{enumerate}
\end{thm}

This theorem has some predecessors. The implications $(i)\Rightarrow(iii)$ and $(iii),\;r<1/4\Rightarrow(i)$ 
for $\alpha=1$ are due to Ghatage, Yan and Zheng by \cite[Proposition 1, p.\ 2040]{ghatage2001} 
and \cite[Theorem 2, p.\ 2043]{ghatage2001}. 
This was improved by Chen to $(i)\Leftrightarrow(iii)$ for $\alpha=1$ by removing the restriction $r<1/4$ 
in \cite[Theorem 1, p.\ 840]{chen2003}. 
The proof of the equivalence $(i)\Leftrightarrow(ii)$ given in \cite[Theorem 1, p.\ 1372]{ghatage2005} for $\alpha=1$ 
is due to Ghatage, Zheng and Zorboska. 
A non-trivial example of a sampling set for $\alpha=1$ can be found in \cite[Example 2, p.\ 1376]{ghatage2005} 
(cf.\ \cite[p.\ 203]{chengauthier2008}).
In the case of several complex variables a characterisation corresponding to \prettyref{thm:Bloch_sampling}
is given by Chen in \cite[Theorem 2, p.\ 844]{chen2003} and Deng, Jiang and Ouyang 
in \cite[Theorem 1-3, p.\ 1031--1032, 1034]{deng2007} where $\Omega$ is the unit ball of $\C^{d}$.
Gim{\'e}nez, Malav{\'e} and Ramos-Fern{\'a}ndez extend \prettyref{thm:Bloch_sampling} by 
\cite[Theorem 3, p.\ 112]{gimenez2010} and \cite[Corollary 1, p.\ 113]{gimenez2010} to more general weights 
of the form $\nu(z)=\mu(1-|z|^{2})$ with some continuous function $\mu\colon (0,1]\to(0,\infty)$ 
such that $\mu(r)\to 0$, $r\to 0\vcenter{\hbox{${\scriptstyle{+}}$}}$, 
which can be extended to a holomorphic function $\mu_{0}$ on $\D(1,1):=\{z\in\C\;|\;|z-1|<1\}$ without zeros 
in $\D(1,1)$ and fulfilling $\mu(1-|1-z|)\leq C|\mu_{0}(z)|$ for all $z\in\D(1,1)$ and some $C>0$ 
(see \cite[p.\ 109]{gimenez2010}). 
Examples of such functions $\mu$ are $\mu_{1}(r):=r^{\alpha}$, $\alpha>0$, $\mu_{2}:=r\ln(2/r)$ 
and $\mu_{3}(r):=r^{\beta}\ln(1-r)$, $\beta>1$, for $r\in(0,1]$ (see \cite[p.\ 110]{gimenez2010}) 
and with $\nu(z)=\mu_{1}(1-|z|^{2})=(1-|z|^{2})^{\alpha}$, $z\in\D$, one gets \prettyref{thm:Bloch_sampling} 
back for $\alpha\geq 1$. For $0<\alpha<1$ and $\nu(z)=\mu_{1}(1-|z|^{2})$, $z\in\D$, 
the equivalence $(i)\Leftrightarrow(ii)$ is given in \cite[Proposition 4.4, p.\ 14]{yoneda2018} of Yoneda as well and 
a sufficient condition implying $(ii)$ in \cite[Corollary 4.5, p.\ 15]{yoneda2018}.
Ramos-Fern{\'a}ndez generalises the results given in \cite{gimenez2010} to bounded essential weights $\nu$ on $\D$ by 
\cite[Theorem 4.3, p.\ 85]{fernandez2011} and \cite[Remark 4.2, p.\ 84]{fernandez2011}.
In \cite[Theorem 2.4, p.\ 3106]{pirasteh2018} Pirasteh, Eghbali and Sanatpour use sets that fix the topology 
in $\mathcal{B}\nu(\D)$ for radial essential $\nu$ to characterise Li-Stevi\'{c} integral-type operators 
on $\mathcal{B}\nu(\D)$ with closed range instead of composition operators. 
The composition operator on the harmonic variant of the Bloch type space $\mathcal{B}\nu(\D)$ 
with $\nu(z)=(1-|z|^{2})^{\alpha}$, $z\in\D$, for some $\alpha>0$ is considered by Esmaeili, Estaremi and Ebadian, 
who give a corresponding result in \cite[Theorem 2.8, p.\ 542]{esmaeili2018}.
\section{Extension of sequentially bounded functions}
In this section we restrict to the case that $E$ is a Fr\'echet space.

\begin{defn}[{\cite[Definition 12, p.\ 8]{B/F/J}}]
 Let $E$ be a Fr\'{e}chet space. An increasing sequence $(B_{n})_{n\in\N}$ of bounded subsets of $E_{b}'$ fixes 
 the topology in $E$ if $(B_{n}^{\circ})_{n\in\N}$ is a fundamental system of zero neighbourhoods of $E$.
\end{defn}

In particular, if $E$ is a Banach space, then an \emph{almost norming} set $B\subset E'$, 
i.e.\ $B$ is bounded w.r.t.\ to the operator norm and the polar $B^{\circ}$ is bounded in $E$, 
fixes the topology in $E$ and we refer the reader to 
\cite[Remark 1.2, p.\ 780--781]{Arendt2000} for examples of such sets. We recall the following special case 
of \cite[5.1 Definition, p.\ 27]{kruse2018_3}.

\begin{defn}[{$sb$-restriction space}]
Let $E$ be a Fr\'{e}chet space, $(B_{n})$ fix the topology in $E$ and $G:=\operatorname{span}(\bigcup_{n\in\N} B_{n})$. 
Let $\FV$ be a Hausdorff space, $U$ a set of uniqueness for $(T^{\K},\mathcal{F}\nu)$ and
set
\[
\mathcal{FV}_{G}(U,E)_{sb}:=\{f\in\mathcal{FV}_{G}(U,E)\;|\;\forall\;n\in\N:\; \{f_{e'}\;|\;e'\in B_{n}\}\;
\text{is bounded in}\;\FV\}.
\]
\end{defn}

Let $E$ be a Fr\'{e}chet space, $(B_{n})$ fix the topology in $E$ and recall the assumptions of 
\prettyref{rem:R_well-defined}.
Let $(T^{E},T^{\K})$ be a strong, consistent family for $(\mathcal{F},E)$ and a generator for $(\mathcal{FV},E)$. 
Let $\F$ and $\FE$ be $\varepsilon$-into-compatible and the inclusion $\FV\hookrightarrow\F$ continuous.
Consider a set of uniqueness $U$ for $(T^{\K},\mathcal{F}\nu)$ 
and $G:=\operatorname{span}(\bigcup_{n\in\N} B_{n})\subset E'$.
For $u\in\F\varepsilon E$ such that $u(B_{\FV}^{\circ \F'})$ is bounded in $E$
we have $R_{U,G}(f)\in\mathcal{F}\nu_{G}(U,E)$ with $f:=S(u)\in\Feps$ by \eqref{eq:well_def_B_unique}. 
We note that
\[
 \sup_{e'\in B_{n}}|f_{e'}|_{\FV}
=\sup_{e'\in B_{n}}\sup_{x\in\omega}|e'(T^{E}(f)(x)\nu(x))|
=\sup_{e'\in B_{n}}\sup_{y\in N_{\omega}(f)}|e'(y)|
\]
with the bounded set $N_{\omega}(f):=\{T^{E}(f)(x)\nu(x)\;|\;x\in\omega\}\subset E$, 
implying $R_{U,G}(f)\in\mathcal{FV}_{G}(U,E)_{sb}$. 
Thus the injective linear map
\[
R_{U,G}\colon \Feps \to \mathcal{F}\nu_{G}(U,E)_{sb},\;f\mapsto (T^{E}(f)(x))_{x\in U},
\]
is well-defined. 

\begin{que}
Let the assumptions of \prettyref{rem:R_well-defined} be fulfilled, 
$E$ be a Fr\'{e}chet space, $(B_{n})$ fix the topology in $E$ and $G:=\operatorname{span}(\bigcup_{n\in\N} B_{n})$. 
When is the injective restriction map 
\[
R_{U,G}\colon \Feps \to \mathcal{F}\nu_{G}(U,E)_{sb},\;f\mapsto (T^{E}(f)(x))_{x\in U},
\]
surjective?
\end{que}

Now, we can generalise \cite[Corollary 2.4, p.\ 692]{F/J/W} and \cite[Theorem 11, p.\ 5]{jorda2013}.

\begin{cor}\label{cor:ext_B_unique_seq_bound}
Let $E$ be a Fr\'{e}chet space, $(B_{n})$ fix the topology in $E$, 
set $G:=\operatorname{span}(\bigcup_{n\in\N} B_{n})$ and 
let $\F$ and $\FE$ be $\varepsilon$-into-compatible. 
Let $(T^{E},T^{\K})$ be a generator for $(\mathcal{F}\nu,E)$ and a strong, consistent family for $(\mathcal{F},E)$,
$\FV$ a Banach space whose closed unit ball $B_{\FV}$ is a compact subset of $\F$ 
and $U$ a set of uniqueness for $(T^{\K},\mathcal{F}\nu)$.
Then the restriction map 
\[
 R_{U,G}\colon\Feps \to \mathcal{F}\nu_{G}(U,E)_{sb} 
\]
is surjective.
\end{cor}
\begin{proof}
Let $f\in\mathcal{F}\nu_{G}(U,E)_{sb}$. Then $\{f_{e'}\;|\;e'\in B_{n}\}$ is bounded in $\FV$ for each $n\in\N$. 
We deduce for each $n\in\N$, $(a_{k})_{k\in\N}\in\ell^{1}$ and $(e_{k}')_{k\in\N}\subset B_{n}$ that 
$(\sum_{k\in\N}a_{k}e_{k}')\circ f$ admits the extension $\sum_{k\in\N}a_{k}f_{e_{k}'}$ in $\FV$. 
Due to \cite[Proposition 7, p.\ 503]{F/J} the LB-space 
$E'((B_{n})_{n\in\N}):=\operatorname{ind}_{n\in\N}E'(B_{n})$, where 
\[
E'(B_{n}):=\{\sum_{k\in\N}a_{k}e_{k}'\;|\;(a_{k})_{k\in\N}\in\ell^{1},\,(e_{k}')_{k\in\N}\subset B_{n}\}
\]
is endowed with its Banach space topology for $n\in\N$, determines boundedness in $E$. 
Hence we conclude that $f\in \mathcal{F}\nu_{E'((B_{n})_{n\in\N})}(U,E)$, which yields that there 
is $u\in\F\varepsilon E$ with bounded $u(B_{\FV}^{\circ \F'})\subset E$ 
such that $R_{U,G}(S(u))=f$ by \prettyref{thm:ext_B_unique}.
\end{proof}

As an application we directly obtain the following two corollaries of \prettyref{cor:ext_B_unique_seq_bound} 
since its assumptions are fulfilled by the proof of 
\prettyref{cor:hypo_weighted_ext_unique} and \prettyref{cor:Bloch_ext_unique}, respectively.

\begin{cor}\label{cor:hypo_weighted_ext_unique_seq_bound}
Let $E$ be a Fr\'{e}chet space, $(B_{n})$ fix the topology in $E$ and $G:=\operatorname{span}(\bigcup_{n\in\N} B_{n})$,
$\Omega\subset\R^{d}$ open, $P(\partial)^{\K}$ a hypoelliptic linear partial differential operator, 
$\nu\colon\Omega\to(0,\infty)$ continuous and $U$ a set of uniqueness 
for $(\id_{\K^{\Omega}},\mathcal{C}\nu^{\infty}_{P(\partial)})$.
If $f\colon U\to E$ is a function such that $e'\circ f$ admits an extension 
$f_{e'}\in\mathcal{C}\nu^{\infty}_{P(\partial)}(\Omega)$ for each $e'\in G$ and $\{f_{e'}\;|\;e'\in B_{n}\}$ is 
bounded in $\mathcal{C}\nu^{\infty}_{P(\partial)}(\Omega)$ for each $n\in\N$, 
then there exists a unique extension $F\in\mathcal{C}\nu^{\infty}_{P(\partial)}(\Omega,E)$ of $f$.
\end{cor}

\begin{cor}\label{cor:Bloch_ext_unique_seq_bound}
Let $E$ be a Fr\'{e}chet space, $(B_{n})$ fix the topology in $E$ and $G:=\operatorname{span}(\bigcup_{n\in\N} B_{n})$, 
$\nu\colon\D\to(0,\infty)$ continuous and $U_{\ast}\subset\D$ have an accumulation point in $\D$. 
If $f\colon \{0\}\cup(\{1\}\times U_{\ast})\to E$ is a function such that there is 
$f_{e'}\in\mathcal{B}\nu(\D)$ for each $e'\in G$ with $f_{e'}(0)=e'(f(0))$ and $f_{e'}'(z)=e'(f(1,z))$ 
for all $z\in U_{\ast}$ and $\{f_{e'}\;|\;e'\in B_{n}\}$ is bounded in $\mathcal{B}\nu(\D)$ for each $n\in\N$, 
then there exists a unique $F\in\mathcal{B}\nu(\D,E)$ with $F(0)=f(0)$ and $(\partial_{\C}^{1})^{E}F(z)=f(1,z)$ 
for all $z\in U_{\ast}$.
\end{cor}
\section{Weak-strong principles for differentiable functions of finite order}
This section is dedicated to $\mathcal{C}^{k}$-weak-strong principles for differentiable functions. So the question is:

\begin{que}
Let $E$ be an lcHs, $G\subset E'$ a separating subspace, $\Omega\subset\R^{d}$ open 
and $k\in\N_{0}\cup\{\infty\}$. If $f\colon\Omega\to E$ is such that $e'\circ f\in\mathcal{C}^{k}(\Omega)$ 
for each $e'\in G$, does $f\in\mathcal{C}^{k}(\Omega,E)$ hold?
\end{que}

An affirmative answer to the preceding question is called a $\mathcal{C}^{k}$-weak-strong principle. 
It is a result of Bierstedt \cite[2.10 Lemma, p.\ 140]{B2} that for $k=0$ the $\mathcal{C}^{0}$-weak-strong 
principle holds if $\Omega\subset\R^{d}$ is open (or more general a $k_{\R}$-space), $G=E'$ and 
$E$ is such that every bounded set is already precompact in $E$. For instance, the last condition is fulfilled if $E$ is 
a semi-Montel or Schwartz space. The $\mathcal{C}^{0}$-weak-strong principle does not hold for general $E$ by 
\cite[Beispiel, p.\ 232]{Kaballo}.

Grothendieck sketches in a footnote \cite[p.\ 39]{Grothendieck1953} 
(cf.\ \cite[Chap.\ 3, Sect.\ 8, Corollary 1, p.\ 134]{gr73}) 
the proof that for $k<\infty$ a weakly-$\mathcal{C}^{k+1}$ function $f\colon\Omega\to E$ on an open set 
$\Omega\subset\R^{d}$ with values in a quasi-complete lcHs $E$ is already $\mathcal{C}^{k}$, i.e.\ that from 
$e'\circ f\in\mathcal{C}^{k+1}(\Omega)$ for all $e'\in E'$ it follows $f\in\mathcal{C}^{k}(\Omega,E)$. 
A detailed proof of this statement is given by Schwartz in \cite{Schwartz1955}, simultaneously weakening 
the condition from quasi-completeness of $E$ to sequential completeness and from 
weakly-$\mathcal{C}^{k+1}$ to weakly-$\mathcal{C}^{k,1}_{loc}$.

\begin{thm}[{\cite[Appendice, Lemme II, Remarques 1$^0$), p.\ 146-147]{Schwartz1955}}]\label{thm:schwartz_weak_strong}
Let $E$ be a sequentially complete lcHs, $\Omega\subset\R^{d}$ open and $k\in\N_{0}$. 
\begin{enumerate}
\item [a)] If $f\colon\Omega\to E$ is such that $e'\circ f\in \mathcal{C}^{k,1}_{loc}(\Omega)$ for all $e'\in E'$, 
then $f\in\mathcal{C}^{k}(\Omega,E)$. 
\item [b)] If $f\colon\Omega\to E$ is such that $e'\circ f\in \mathcal{C}^{k+1}(\Omega)$ for all $e'\in E'$, 
then $f\in\mathcal{C}^{k}(\Omega,E)$. 
\end{enumerate}
\end{thm} 

Here $\mathcal{C}^{k,1}_{loc}(\Omega)$ denotes the space of functions in $\mathcal{C}^{k}(\Omega)$ 
whose partial derivatives of order $k$ are locally Lipschitz continuous. 
Part b) clearly implies a $\mathcal{C}^{\infty}$-weak-strong principle for 
open $\Omega\subset\R^{d}$, $G=E'$ and sequentially complete $E$. This can be generalised to locally complete $E$.
Waelbroeck has shown in \cite[Proposition 2, p.\ 411]{waelbroeck1967_2} and 
\cite[Definition 1, p.\ 393]{waelbroeck1967_1} 
that the $\mathcal{C}^{\infty}$-weak-strong principle holds if $\Omega$ is a manifold, $G=E'$ 
and $E$ is locally complete.
It is a result of Bonet, Frerick and Jord\'a that the $\mathcal{C}^{\infty}$-weak-strong principle still holds by 
\cite[Theorem 9, p.\ 232]{B/F/J} if $\Omega\subset\R^{d}$ is open, $G\subset E'$ determines boundedness 
and $E$ is locally complete.
Due to \cite[2.14 Theorem, p.\ 20]{kriegl} of Kriegl and Michor an lcHs $E$ is locally complete if and 
only if the $\mathcal{C}^{\infty}$-weak-strong principle holds for $\Omega=\R$ and $G=E'$. 

One of the goals of this section is to improve \prettyref{thm:schwartz_weak_strong}. 
We start with the following definition. For $k\in\N_{0}$ we define the space of $k$-times continuously 
partially differentiable $E$-valued functions on an open set $\Omega\subset\R^{d}$ whose partial derivatives 
up to order $k$ are continuously extendable to the boundary of $\Omega$ by 
\[
 \mathcal{C}^{k}(\overline{\Omega},E):=\{f\in\mathcal{C}^{k}(\Omega,E)\;|\;(\partial^{\beta})^{E}f\;
 \text{cont.\ extendable on}\;\overline{\Omega}\;\text{for all}\;\beta\in\N^{d}_{0},\,|\beta|\leq k\}
\]
which we equip with the system of seminorms given by 
\[
 |f|_{\mathcal{C}^{k}(\overline{\Omega}),\alpha}:=\sup_{\substack{x\in \Omega\\ \beta\in\N^{d}_{0}, |\beta|\leq k}}
 p_{\alpha}((\partial^{\beta})^{E}f(x)), \quad f\in\mathcal{C}^{k}(\overline{\Omega},E),\, \alpha\in\mathfrak{A}.
\]
The space of functions in $\mathcal{C}^{k}(\overline{\Omega},E)$ such 
that all its $k$-th partial derivatives are $\gamma$-H\"older continuous with $0<\gamma\leq 1$ is given by  
\[
\mathcal{C}^{k,\gamma}(\overline{\Omega},E):=
\bigl\{f\in\mathcal{C}^{k}(\overline{\Omega},E)\;|\;\forall\;\alpha\in\mathfrak{A}:\;
|f|_{\mathcal{C}^{k,\gamma}(\overline{\Omega}),\alpha}<\infty\bigr\}
\]
where
\[
|f|_{\mathcal{C}^{k,\gamma}(\overline{\Omega}),\alpha}:=\max\Bigl(|f|_{\mathcal{C}^{k}(\overline{\Omega}),\alpha},
\sup_{\beta\in\N^{d}_{0}, |\beta|=k}|(\partial^{\beta})^{E}f|_{\mathcal{C}^{0,\gamma}(\Omega),\alpha}\Bigr)
\]
with
\[
|f|_{\mathcal{C}^{0,\gamma}(\Omega),\alpha}:=\sup_{\substack{x,y\in\Omega\\x\neq y}}
\frac{p_{\alpha}(f(x)-f(y))}{|x-y|^{\gamma}}.
\]
We set 
\[
\omega_{1}:=\{\beta\in\N_{0}^{d}\;|\;|\beta|\leq k\}\times\Omega\quad\text{and}\quad
\omega_{2}:=\{\beta\in\N_{0}^{d}\;|\;|\beta|=k\}\times (\Omega^{2}\setminus\{(x,x)\;|\;x\in\Omega\})
\]
as well as $\omega:=\omega_{1}\cup\omega_{2}$. We define the operator 
$T^{E}\colon\mathcal{C}^{k}(\Omega,E)\to E^{\omega}$ by
\begin{align*}
T^{E}(f)(\beta,x):=&(\partial^{\beta})^{E}(f)(x) &&,\;(\beta,x)\in\omega_{1},\\
T^{E}(f)(\beta,(x,y)):=&(\partial^{\beta})^{E}(f)(x)-(\partial^{\beta})^{E}(f)(y) &&,\;(\beta,(x,y))\in\omega_{2}.
\end{align*}
and the weight $\nu\colon\omega\to (0,\infty)$ by 
\[
\nu(\beta,x):=1,\;(\beta,x)\in\omega_{1},\quad\text{and}\quad \nu(\beta,(x,y)):=\frac{1}{|x-y|^{\gamma}},
\;(\beta,(x,y))\in\omega_{2}.
\]
By setting $\FE:=\mathcal{C}^{k}(\overline{\Omega},E)$ and observing that
\[
|f|_{\mathcal{C}^{k,\gamma}(\overline{\Omega}),\alpha}=\sup_{x\in\omega}p_{\alpha}(T^{E}(f)(x))\nu(x),
\quad f\in\mathcal{C}^{k,\gamma}(\overline{\Omega},E),\,\alpha\in\mathfrak{A},
\]
we have $\mathcal{F}\nu(\Omega,E)=\mathcal{C}^{k,\gamma}(\overline{\Omega},E)$ with generator $(T^{E},T^{\K})$. 

\begin{cor}\label{cor:ext_B_unique}
Let $E$ be a locally complete lcHs, $G\subset E'$ determine boundedness, $\Omega\subset\R^{d}$ open and bounded, 
$k\in\N_{0}$ and $0<\gamma\leq 1$. In the case $k\geq 1$, assume additionally that $\Omega$ has Lipschitz boundary. 
If $f\colon\Omega\to E$ is such that $e'\circ f\in \mathcal{C}^{k,\gamma}(\overline{\Omega})$ for all $e'\in G$, 
then $f\in\mathcal{C}^{k,\gamma}(\overline{\Omega},E)$. 
\end{cor}
\begin{proof}
We take $\F:=\mathcal{C}^{k}(\overline{\Omega})$ and $\FE:=\mathcal{C}^{k}(\overline{\Omega},E)$ 
and have $\mathcal{F}\nu(\Omega)=\mathcal{C}^{k,\gamma}(\overline{\Omega})$ 
and $\mathcal{F}\nu(\Omega,E)=\mathcal{C}^{k,\gamma}(\overline{\Omega},E)$ with the weight 
$\nu$ and generator $(T^{E},T^{\K})$ for $(\mathcal{F}\nu,E)$ described above.
Due to the proof of \cite[Example 20, p.\ 1529]{kruse2017} and the first part of the proof of 
\cite[Theorem 14, p.\ 1524]{kruse2017} the spaces $\F$ and $\FE$ are 
$\varepsilon$-into-compatible for any lcHs $E$ 
(the condition that $E$ has metric ccp in \cite[Example 20, p.\ 1529]{kruse2017} is only needed for 
$\varepsilon$-compatibility). 
Another consequence of \cite[Example 20, p.\ 1529]{kruse2017} is that 
\[
T^{E}(S(u))(\beta,x)=(\partial^{\beta})^{E}(S(u))(x)=u(\delta_{x}\circ(\partial^{\beta})^{\K})=u(T^{\K}_{\beta,x}),
\quad (\beta,x)\in\omega_{1},
\]
holds for all $u\in\F\varepsilon E$, implying 
\begin{align*}
T^{E}(S(u))(\beta,(x,y))&=T^{E}(S(u))(\beta,x)-T^{E}(S(u))(\beta,y)=u(T^{\K}_{\beta,x})-u(T^{\K}_{\beta,y})\\
&=u(T^{\K}_{\beta,(x,y)}),\quad(\beta,(x,y))\in\omega_{2}.
\end{align*}
Thus $(T^{E},T^{\K})$ is a consistent family for $(\mathcal{F},E)$ and its strength is easily seen.
In addition, $\mathcal{F}\nu(\Omega)=\mathcal{C}^{k,\gamma}(\overline{\Omega})$ is a Banach space 
by \cite[Theorem 9.8, p.\ 110]{driver2004} (cf.\ \cite[1.7 H\"olderstetige Funktionen, p.\ 46]{alt2012}) 
whose closed unit ball is compact in
$\F=\mathcal{C}^{k}(\overline{\Omega})$ by \cite[8.6 Einbettungssatz in H\"older-R\"aumen, p.\ 338]{alt2012}. 
Moreover, the $\varepsilon$-into-compatibility of $\F$ and $\FE$ 
in combination with the consistency of $(T^{E},T^{\K})$ for $(\mathcal{F},E)$ implies 
$\Feps\subset\mathcal{F}\nu(\Omega,E)$ as linear spaces by \prettyref{prop:mingle-mangle} c). 
Hence our statement follows from \prettyref{thm:ext_B_unique}
with the set of uniqueness $U:=\{0\}\times\Omega$ for $(T^{\K},\mathcal{F}\nu)$.
\end{proof}

\begin{rem}
We point out that \prettyref{cor:ext_B_unique} corrects our result \cite[Corollary 5.3, p.\ 16]{kruse2019_3} 
by adding the missing assumption that $\Omega$ should additionally
have Lipschitz boundary in the case $k\geq 1$. This is needed to deduce that the closed unit ball of 
$\mathcal{C}^{k,\gamma}(\overline{\Omega})$ is compact in $\mathcal{C}^{k}(\overline{\Omega})$ 
by \cite[8.6 Einbettungssatz in H\"older-R\"aumen, p.\ 338]{alt2012} (in the notation of \cite{gilbarg_trudinger2001} 
$\Omega$ having Lipschitz boundary means that it is a $\mathcal{C}^{0,1}$ domain, see 
\cite[Lemma 6.36, p.\ 136]{gilbarg_trudinger2001} and the comments below and above this lemma). 
This additional assumption is missing in \cite[Theorem 14.32, p.\ 232]{driver2004}, 
which is our main reference in \cite{kruse2019_3} for the compact embedding, 
but it is needed due to \cite[U8.1 Gegenbeispiel zu Einbettungss\"atzen, p.\ 365]{alt2012} (cf.\ \cite[p.\ 53]{gilbarg_trudinger2001}).  
However, this only affects the result \cite[Corollary 6.3, p.\ 21--22]{kruse2019_3} 
where we have to add this missing assumption as well 
(see \prettyref{cor:Hoelder_Blaschke} for this). The other results of \cite{kruse2019_3} 
derived from \cite[Corollary 5.3, p.\ 16]{kruse2019_3} are not affected by this missing assumption since 
they are all a consequence of \cite[Corollary 5.4, p.\ 17]{kruse2019_3} and \cite[Corollary 6.4, p.\ 22]{kruse2019_3}, 
whose proofs can be adjusted without additional 
assumptions (see \prettyref{cor:ext_B_unique_loc_Hoelder} and \prettyref{cor:loc_Hoelder_Blaschke} for this).
\end{rem}

Next, we use the preceding corollary to generalise the theorem of Grothendieck and Schwartz on weakly 
$\mathcal{C}^{k+1}$-functions.
For $k\in\N_{0}$ and $0<\gamma\leq 1$ we define the space of $k$-times continuously partially differentiable 
$E$-valued functions with locally $\gamma$-H\"older continuous partial derivatives of $k$-th order
on an open set $\Omega\subset\R^{d}$ by
\[
 \mathcal{C}^{k,\gamma}_{loc}(\Omega,E):=\{f\in\mathcal{C}^{k}(\Omega,E)\;|\;
 \forall\;K\subset\Omega\;\text{compact},\,\alpha\in\mathfrak{A}:\;|f|_{K,\alpha}<\infty\}
\]
where 
\[
|f|_{K,\alpha}:=\max\Bigl(|f|_{\mathcal{C}^{k}(K),\alpha},\,\sup_{\beta\in\N^{d}_{0}, |\beta|= k}
|(\partial^{\beta})^{E}f|_{\mathcal{C}^{0,\gamma}(K),\alpha}\Bigr)
\]
with
\[
 |f|_{\mathcal{C}^{k}(K),\alpha}:=\sup_{\substack{x\in K\\ \beta\in\N^{d}_{0}, |\beta|\leq k}}
 p_{\alpha}((\partial^{\beta})^{E}f(x))
\]
and 
\[
|f|_{\mathcal{C}^{0,\gamma}(K),\alpha}:=\sup_{\substack{x,y\in K\\x\neq y}}
\frac{p_{\alpha}(f(x)-f(y))}{|x-y|^{\gamma}}.
\]

Using \prettyref{cor:ext_B_unique}, we are able to improve \prettyref{thm:schwartz_weak_strong} to the following form.  

\begin{cor}\label{cor:ext_B_unique_loc_Hoelder}
Let $E$ be a locally complete lcHs, $G\subset E'$ determine boundedness, $\Omega\subset\R^{d}$ open, 
$k\in\N_{0}$ and $0<\gamma\leq 1$. 
\begin{enumerate}
\item [a)] If $f\colon\Omega\to E$ is such that $e'\circ f\in \mathcal{C}^{k,\gamma}_{loc}(\Omega)$ for all $e'\in G$, 
then $f\in\mathcal{C}^{k,\gamma}_{loc}(\Omega,E)$. 
\item [b)] If $f\colon\Omega\to E$ is such that $e'\circ f\in \mathcal{C}^{k+1}(\Omega)$ for all $e'\in G$, 
then $f\in\mathcal{C}^{k,1}_{loc}(\Omega,E)$. 
\end{enumerate}
\end{cor}
\begin{proof}
Let us start with a). Let $f\colon\Omega\to E$ be such that $e'\circ f\in \mathcal{C}^{k,\gamma}_{loc}(\Omega)$ 
for all $e'\in G$. Let $(\Omega_{n})_{n\in\N}$ be an exhaustion of $\Omega$ with open, relatively compact sets 
$\Omega_{n}\subset\Omega$ with Lipschitz boundaries $\partial\Omega_{n}$ (e.g.\ choose each $\Omega_{n}$ as the interior 
of a finite union of closed axis-parallel cubes, see the proof of \cite[Theorem 1.4, p.\ 7]{stein2005} for the construction) 
that satisfies $\Omega_{n}\subset\Omega_{n+1}$ for all $n\in\N$. 
Then the restriction of $e'\circ f$ to $\Omega_{n}$ is an element of $\mathcal{C}^{k,\gamma}(\overline{\Omega}_{n})$ 
for every $e'\in G$ and $n\in\N$. Due to \prettyref{cor:ext_B_unique} we obtain that 
$f\in\mathcal{C}^{k,\gamma}(\overline{\Omega}_{n},E)$ for every $n\in\N$. 
Thus $f\in\mathcal{C}^{k,\gamma}_{loc}(\Omega,E)$ 
since differentiability is a local property and for each compact $K\subset\Omega$ there is $n\in\N$ 
such that $K\subset \Omega_{n}$.

Let us turn to b), i.e.\ let $f\colon\Omega\to E$ be such that $e'\circ f\in\mathcal{C}^{k+1}(\Omega)$ 
for all $e'\in G$. 
Since $\Omega\subset\R^{d}$ is open, for every $x\in\Omega$ there is $\varepsilon_{x}>0$ such that 
$\overline{\mathbb{B}(x,\varepsilon_{x})}\subset\Omega$ where $\overline{\mathbb{B}(x,\varepsilon_{x})}$ 
is the closure of $\mathbb{B}(x,\varepsilon_{x}):=\{y\in\R^{d}\;|\;|y-x|<\varepsilon_{x}\}$.
For all $e'\in G$, $\beta\in\N_{0}^{d}$ with $|\beta|=k$ 
and $w,y\in\overline{\mathbb{B}(x,\varepsilon_{x})}$, $w\neq y$, it holds that
\[
 \frac{|(\partial^{\beta})^{\K}(e'\circ f)(w)-(\partial^{\beta})^{\K}(e'\circ f)(y)|}{|w-y|}
\leq C_{d}\max_{1\leq n\leq d}
 \max_{z\in\overline{\mathbb{B}(x,\varepsilon_{x})}}|(\partial^{\beta+e_{n}})^{\K}(e'\circ f)(z)|
\]
by the mean value theorem applied to the real and imaginary part 
where $C_{d}:=\sqrt{d}$ if $\K=\R$ and $C_{d}:=2\sqrt{d}$ if $\K=\C$. 
Thus $e'\circ f\in \mathcal{C}^{k,1}_{loc}(\Omega)$ for all $e'\in G$ 
since for each compact set $K\subset\Omega$ there are $m\in\N$ and $x_{i}\in\Omega$, $1\leq i\leq m$, 
such that $K\subset \bigcup_{i=1}^{m}\mathbb{B}(x_{i},\varepsilon_{x_{i}})$.
It follows from part a) that $f\in\mathcal{C}^{k,1}_{loc}(\Omega,E)$. 
\end{proof}

If $\Omega=\R$, $\gamma=1$ and $G=E'$, then part a) of \prettyref{cor:ext_B_unique_loc_Hoelder} is already known 
by \cite[2.3 Corollary, p.\ 15]{kriegl}. 
A `full' $\mathcal{C}^{k}$-weak-strong principle for $k<\infty$, i.e.\ the conditions of part b) 
imply $f\in\mathcal{C}^{k+1}(\Omega,E)$, does not hold in general (see \cite[p.\ 11--12]{kriegl}) 
but it holds if we restrict the class of admissible lcHs $E$.
%Let $E$ be an infinite dimensional Hilbert space with orthonormal basis $(\mathrm{e}_{n})_{n\in\N}$. 
%Let $(\varphi_{n})_{n\in\N}$ be a sequence in $\mathcal{C}^{k}_{c}(\R)$ with $\operatorname{supp}\varphi_{n}\subset (\frac{1}{n+1},\frac{1}{n})=:I_{n}$, 
%$0\leq\varphi_{n}\leq 1$ and $\varphi(t_{n})=1$ for some $t_{n}\in I_{n}$ and all $n\in\N$. 
%We set 
%\[
%f\colon\R\to E,\;f(t):=\sum_{n=1}^{\infty}\varphi_{n}(t)\mathrm{e}_{n},
%\]
%and note that $f\in\mathcal{C}^{k}(\R\setminus\{0\},E)$ with $(\partial^{m})^{E}f=\partial^{m}\varphi_{n}\mathrm{e}_{n}$ on $I_{n}$ for all $m\in\N$ 
%which follows from $f=\varphi_{n}\mathrm{e}_{n}$ on $I_{n}$. Further, $f$ is not continuous in $0$ since $f(t_{n})=e_{n}$ and $t_{n}\to 0$. 
%But for all $t\in I_{n}$ and $x\in E$ holds
%\begin{align*}
%\bigl|\langle \frac{(\partial^{m})^{E}f(t)}{t},x\rangle\bigr|&=\bigl|\langle \frac{\partial^{m}\varphi_{n}(t)\mathrm{e}_{n}}{t},x\rangle\bigr|\\
%&=\bigl|\frac{\partial^{m}\varphi_{n}(t)}{t}\bigr|\cdot|\langle\mathrm{e}_{n},x\rangle|
%\end{align*}
 
\begin{thm}\label{thm:weak_strong_finite_order}
Let $E$ be a semi-Montel space, $G\subset E'$ determine boundedness, $\Omega\subset\R^{d}$ open and
$k\in\N$. If $f\colon\Omega\to E$ is such that $e'\circ f\in \mathcal{C}^{k}(\Omega)$ for all $e'\in G$, 
then $f\in\mathcal{C}^{k}(\Omega,E)$. 
\end{thm}
\begin{proof}
Let $f\colon\Omega\to E$ be such that $e'\circ f\in \mathcal{C}^{k}(\Omega)$ for all $e'\in G$. 
Due to \prettyref{cor:ext_B_unique_loc_Hoelder} b) we already know that $f\in\mathcal{C}^{k-1,1}_{loc}(\Omega,E)$ since 
semi-Montel spaces are quasi-complete and thus locally complete. 
Now, let $x\in\Omega$, $\varepsilon_{x}>0$ such that $\overline{\mathbb{B}(x,\varepsilon_{x})}\subset\Omega$, 
$\beta\in\N_{0}^{d}$ with $|\beta|=k-1$ and $n\in\N$ with $1\leq n\leq d$. 
The set
\[
B:=\Bigl\{\frac{(\partial^{\beta}f)^{E}(x+he_{n})-(\partial^{\beta}f)^{E}f(x)}{h}\;|\;
          h\in\R,\,0<h\leq\varepsilon_{x}\Bigr\}
\]
is bounded in $E$ because $f\in\mathcal{C}^{k-1,1}_{loc}(\Omega,E)$. As $E$ is semi-Montel, 
the closure $\overline{B}$ is compact in $E$. 
Let $(h_{m})_{m\in\N}$ be a sequence in $\R$ such that $0<h_{m}\leq\varepsilon_{x}$ for all $m\in\N$. 
From the compactness of $\overline{B}$ we deduce that there is a subnet $(h_{m_{\iota}})_{\iota\in I}$, 
where $I$ is a directed set, of $(h_{m})_{m\in\N}$ and $y_{x}\in\overline{B}$ with 
\[
y_{x}=\lim_{\iota\in I}\frac{(\partial^{\beta}f)^{E}(x+h_{m_{\iota}}e_{n})-(\partial^{\beta}f)^{E}f(x)}{h_{m_{\iota}}}
=:\lim_{\iota\in I}y_{\iota}.
\]
Further, we note that the limit
\begin{equation}\label{eq:weak_limit_1}
 \partial^{\beta+e_{n}}(e'\circ f)(x)
=\lim_{\substack{h\to 0\\ h\in\R, h\neq 0}}
 \frac{\partial^{\beta}(e'\circ f)(x+he_{n})-\partial^{\beta}(e'\circ f)(x)}{h}
\end{equation}
exists for all $e'\in G$ and that $(e'(y_{\iota}))_{\iota\in I}$ is a subnet of the net of difference quotients 
on the right-hand side of \eqref{eq:weak_limit_1} as $\partial^{\beta}(e'\circ f)=e'\circ (\partial^{\beta})^{E} f$.
Therefore 
\begin{align}\label{eq:weak_limit_2}
  \partial^{\beta+e_{n}}(e'\circ f)(x)
&=\lim_{\substack{h\to 0\\ h\in\R, h\neq 0}}
   e'\Bigl(\frac{(\partial^{\beta})^{E}f(x+he_{n})-(\partial^{\beta})^{E}f(x)}{h}\Bigr)\notag\\
&=\lim_{\substack{h\to 0\\ h\in\R, 0<h\leq\varepsilon_{x}}}
   e'\Bigl(\frac{(\partial^{\beta})^{E}f(x+he_{n})-(\partial^{\beta})^{E}f(x)}{h}\Bigr)
 =\lim_{\iota\in I}e'(y_{\iota})\notag\\
&=e'(y_{x})
\end{align}
for all $e'\in G$. By \cite[4.10 Proposition (i), p.\ 21]{kruse2018_3} the topology $\sigma(E,G)$ 
and the initial topology of $E$ coincide on $\overline{B}$. Combining this fact with
\eqref{eq:weak_limit_2}, we deduce that 
\[
 (\partial^{\beta+e_{n}})^{E}f(x)
=\lim_{\substack{h\to 0\\ h\in\R, h\neq 0}}\frac{(\partial^{\beta})^{E}f(x+he_{n})-(\partial^{\beta})^{E}f(x)}{h}
=y_{x}.
\]
In addition, $e'\circ(\partial^{\beta+e_{n}})^{E}f=\partial^{\beta+e_{n}}(e'\circ f)$ is continuous on 
$\overline{\mathbb{B}(x,\varepsilon_{x})}$ for all $e'\in G$, meaning that the restriction of 
$(\partial^{\beta+e_{n}})^{E}f$ on $\overline{\mathbb{B}(x,\varepsilon_{x})}$ to $(E,\sigma(E,G))$
is continuous, and the range $(\partial^{\beta+e_{n}})^{E}f(\overline{\mathbb{B}(x,\varepsilon_{x})})$ is bounded 
in $E$. As before we use that $\sigma(E,G)$ and the initial topology of $E$ coincide on 
$(\partial^{\beta+e_{n}})^{E}f(\overline{\mathbb{B}(x,\varepsilon_{x})})$, 
which implies that the restriction of $(\partial^{\beta+e_{n}})^{E}f$ on $\overline{\mathbb{B}(x,\varepsilon_{x})}$ 
is continuous w.r.t.\ the initial topology of $E$. 
Since continuity is a local property and $x\in\Omega$ is arbitrary, we conclude that $(\partial^{\beta+e_{n}})^{E}f$ 
is continuous on $\Omega$.
\end{proof}

In the special case that $\Omega=\R$, $G=E'$ and $E$ is a Montel space, i.e.\ a barrelled semi-Montel space, 
a different proof of the preceding weak-strong principle can be found in the proof of 
\cite[Lemma 4, p.\ 15]{carroll1961}. 
This proof uses the Banach-Steinhaus theorem and needs the barrelledness of the Montel space $E_{b}'$. 
Our weak-strong principle \prettyref{thm:weak_strong_finite_order} does not need the barrelledness of $E$, 
hence can be applied to non-barrelled semi-Montel spaces like 
$E=(\mathcal{C}^{\infty}_{\overline{\partial},b}(\D),\beta)$ where $\beta$ is the strict topology 
(see page~\pageref{page:strict}, \cite[3.14 Proposition, p.\ 12]{kruse2018_3} and 
\cite[3.15 Remark, p.\ 13]{kruse2018_3}).

Besides the `full' $\mathcal{C}^{k}$-weak-strong principle for $k<\infty$ and semi-Montel $E$, part b) 
of \prettyref{cor:ext_B_unique_loc_Hoelder} also suggests an `almost' $\mathcal{C}^{k}$-weak-strong principle 
in terms of \cite[3.1.6 Rademacher's theorem, p.\ 216]{federer1969}, which we prepare next.

\begin{defn}[{generalised Gelfand space}]
We say that an lcHs $E$ is a \emph{generalised Gelfand space}
if every Lipschitz continuous map $f\colon [0,1]\to E$ is differentiable almost everywhere w.r.t\ to 
the one-dimensional Lebesgue measure. 
\end{defn}

If $E$ is a real Fr\'echet space ($\K=\R$), then this definition coincides with 
the definition of a Fr\'echet--Gelfand space given in \cite[Definition 2.2, p.\ 17]{mankiewicz1973}. In particular, 
every real nuclear Fr\'echet lattice (see \cite[Theorem 6, Corollary, p.\ 375, 378]{grosse-erdmann1991}) 
and more general every real Fr\'echet--Montel space is a generalised Gelfand space 
by \cite[Theorem 2.9, p.\ 18]{mankiewicz1973}. 
If $E$ is a Banach space, then this definition coincides with the definition of a Gelfand space given 
in \cite[Definition 4.3.1, p.\ 106-107]{diesteluhl1977} by \cite[Proposition 1.2.4, p.\ 18]{arendt2011}. 
A Banach space is a Gelfand space if and only if it has the Radon--Nikod\'ym 
property by \cite[Theorem 4.3.2, p.\ 107]{diesteluhl1977}. 
Thus separable duals of Banach spaces, reflexive Banach spaces and $\ell^{1}(\Gamma)$ 
for any set $\Gamma$ are generalised Gelfand spaces by \cite[Theorem 3.3.1 (Dunford-Pettis), p.\ 79]{diesteluhl1977}, 
\cite[Corollary 3.3.4 (Phillips), p.\ 82]{diesteluhl1977} and \cite[Corollary 3.3.8, p.\ 83]{diesteluhl1977}. 
The Banach spaces $c_{0}$, $\ell^{\infty}$, $\mathcal{C}([0,1])$, $L^{1}([0,1])$ and $L^{\infty}([0,1])$ do not have the 
Radon-Nikod\'ym property and hence are not generalised Gelfand spaces by \cite[Proposition 1.2.9, p.\ 20]{arendt2011}, 
\cite[Example 1.2.8, p.\ 20]{arendt2011} and \cite[Proposition 1.2.10, p.\ 21]{arendt2011}.

\begin{cor}\label{cor:allmost_weak_strong_finite_order}
Let $E$ be a locally complete generalised Gelfand space, $G\subset E'$ determine boundedness, $\Omega\subset\R$ open and
$k\in\N$. If $f\colon\Omega\to E$ is such that $e'\circ f\in \mathcal{C}^{k}(\Omega)$ for all $e'\in G$, 
then $f\in\mathcal{C}^{k-1,1}_{loc}(\Omega,E)$ and the derivative 
$(\partial^{k})^{E}f(x)$ exists for Lebesgue almost all $x\in\Omega$.
\end{cor}
\begin{proof}
The first part follows from \prettyref{cor:ext_B_unique_loc_Hoelder} b). 
Now, let $[a,b]\subset\Omega$ be a bounded interval. 
We set $F\colon[0,1]\to E$, $F(x):=(\partial^{k-1})^{E}f(a+x(b-a))$. 
%For $x,y\in [0,1]$ and $\alpha\in\mathfrak{A}$ holds
%\begin{align*}
%\frac{p_{\alpha}(F(x)-F(y))}{|x-y|}&=|b-a|\frac{p_{\alpha}(F(x)-F(y))}{|x(b-a)-y(b-a)|}\\
%&=|b-a|\frac{p_{\alpha}((\partial^{k-1})^{E}f(a+x(b-a))-(\partial^{k-1})^{E}f(a+y(b-a)))}{|a+x(b-a)-(a+y(b-a))|}.
%\end{align*}
Then $F$ is Lipschitz continuous as $f\in\mathcal{C}^{k-1,1}_{loc}(\Omega,E)$. 
%Let $z\in[a,b]$ with $z=a+x(b-a)$ and $F$ differentiable in $x\in[0,1]$. Then with $\widetilde{h}:=\tfrac{h}{b-a}$, $h\neq 0$, holds
%\[
%\frac{(\partial^{k-1})^{E}f(z+h)-(\partial^{k-1})^{E}f(z)}{h}=\frac{1}{b-a}\frac{F(x+\tfrac{h}{b-a})-F(x)}{\tfrac{h}{b-a}}
%=\frac{1}{b-a}\frac{F(x+\widetilde{h})-F(x)}{\widetilde{h}}.
%\]
This yields that $F$ is differentiable on $[0,1]$ almost everywhere because $E$ is a generalised Gelfand space, 
implying that $(\partial^{k-1})^{E}f$ is differentiable on $[a,b]$ almost everywhere. 
Since the open set $\Omega\subset\R$ can be written as a countable union of disjoint open intervals 
$I_{n}$, $n\in\N$, and each $I_{n}$ is a countable union of closed bounded intervals 
$[a_{m},b_{m}]$, $m\in\N$, our statement follows from the fact that the countable union of null sets is a null set.
\end{proof}

To the best of our knowledge there are still some open problems for continuously partially differentiable functions 
of finite order. 

\begin{que}
\begin{enumerate}
\item[(i)] Are there other spaces than semi-Montel spaces $E$ for which the `full' 
$\mathcal{C}^{k}$-weak-strong principle \prettyref{thm:weak_strong_finite_order} with $k<\infty$ is true? 
For instance, if $k=0$, then it is still true if $E$ is 
an lcHs such that every bounded set is already precompact in $E$ by \cite[2.10 Lemma, p.\ 140]{B2}. 
Does this hold for $0<k<\infty$ as well?
\item[(ii)] Does the `almost' $\mathcal{C}^{k}$-weak-strong principle \prettyref{cor:allmost_weak_strong_finite_order} 
also hold for $d>1$? 
\item[(iii)] For every $\varepsilon>0$ does there exist a function $g\in\mathcal{C}^{k}(\R,E)$ such that 
$\lambda(\{x\in\Omega\;|\;f(x)\neq g(x)\})<\varepsilon$ in \prettyref{cor:allmost_weak_strong_finite_order} 
where $\lambda$ is the one-dimensional Lebesgue measure. 
In the case that $E=\R^{n}$ this is true by \cite[Theorem 3.1.15, p.\ 227]{federer1969}.
\item[(iv)] Is there a `Radon--Nikod\'ym type' characterisation of generalised Gelfand spaces as in the Banach case? 
\end{enumerate}
\end{que}
\section{Vector-valued Blaschke theorems}
In this section we prove several convergence theorems for Banach-valued functions in the 
spirit of Blaschke's convergence theorem \cite[Theorem 7.4, p.\ 219]{burckel1979} as it is done in 
\cite[Theorem 2.4, p.\ 786]{Arendt2000} and \cite[Corollary 2.5, p.\ 786--787]{Arendt2000} 
for bounded holomorphic functions and more general in \cite[Corollary 4.2, p.\ 695]{F/J/W} 
for bounded functions in the kernel of a hypoelliptic linear partial differential operator. 
Blaschke's convergence theorem says that if $(z_{n})_{n\in\N}\subset\D$ is a sequence of distinct elements
with $\sum_{n\in\N}(1-|z_{n}|)=\infty$ and if $(f_{k})_{k\in\N}$ is a bounded sequence in $H^{\infty}(\D)$ 
such that $(f_{k}(z_{n}))_{k}$ converges in $\C$ for each $n\in\N$, then there is $f\in H^{\infty}(\D)$ such that 
$(f_{k})_{k}$ converges uniformly to $f$ on the compact subsets of $\D$, i.e.\ w.r.t.\ to $\tau_{co}$.

\begin{prop}[{\cite[Proposition 4.1, p.\ 695]{F/J/W}}]\label{prop:Blaschke_operator}
Let $(E,\|\cdot\|)$ be a Banach space, $Z$ a Banach space whose closed unit ball $B_{Z}$ is a 
compact subset of an lcHs $Y$ and let $(\mathsf{A}_{\iota})_{\iota\in I}$ be a net in $Y\varepsilon E$ 
such that
\[
\sup_{\iota\in I}\{\|\mathsf{A}_{\iota}(y)\|\;|\;y\in B_{Z}^{\circ Y'}\}<\infty.
\] 
Assume further that there exists a $\sigma(Y',Z)$-dense subspace $X\subset Y'$ such that 
$\lim_{\iota}\mathsf{A}_{\iota}(x)$ exists for each $x\in X$. Then there is 
$\mathsf{A}\in Y\varepsilon E$ with $\mathsf{A}(B_{Z}^{\circ Y'})$ bounded and 
$\lim_{\iota}\mathsf{A}_{\iota}=\mathsf{A}$ uniformly on the equicontinuous subsets of $Y'$, 
i.e.\ for all equicontinuous $B\subset Y'$ and $\varepsilon>0$ there exists $\varsigma\in I$ such that 
\[
\sup_{y\in B}\|\mathsf{A}_{\iota}(y)-\mathsf{A}(y)\|<\varepsilon
\]
for each $\iota\geq\varsigma$.
\end{prop}

Next, we generalise \cite[Corollary 4.2, p.\ 695]{F/J/W}.

\begin{cor}\label{cor:Blaschke_vector_valued}
Let $(E,\|\cdot\|)$ be a Banach space and $\F$ and $\FE$ be $\varepsilon$-into-compatible. 
Let $(T^{E},T^{\K})$ be a generator for $(\mathcal{F}\nu,E)$ 
and a strong, consistent family for $(\mathcal{F},E)$, $\FV$ a Banach space 
whose closed unit ball $B_{\mathcal{F}\nu(\Omega)}$ 
is a compact subset of $\F$ and $U$ a set of uniqueness for $(T^{\K},\mathcal{F}\nu)$. 

If $(f_{\iota})_{\iota\in I}\subset\Feps$ is a bounded net in $\FVE$ such that 
$\lim_{\iota}T^{E}(f_{\iota})(x)$ exists for all $x\in U$, then there is $f\in\Feps$ such that 
$(f_{\iota})_{\iota\in I}$ converges to $f$ in $\FE$.
\end{cor}
\begin{proof}
We set $X:=\operatorname{span}\{T^{\K}_{x}\;|\;x\in U\}$, $Y:=\F$ and $Z:=\FV$.  
As in the proof of \prettyref{thm:ext_B_unique} we observe that $X$ is $\sigma(Y',Z)$-dense in $Y'$. 
From $(f_{\iota})_{\iota\in I}\subset\Feps$ follows that there are $\mathsf{A}_{\iota}\in\F\varepsilon E$ 
with $S(\mathsf{A}_{\iota})=f_{\iota}$ for all $\iota\in I$. Since $(f_{\iota})_{\iota\in I}$ is a bounded net 
in $\FVE$, we note that 
\begin{align*}
  \sup_{\iota\in I}\sup_{x\in\omega}\|\mathsf{A}_{\iota}(T^{\K}_{x}(\cdot)\nu(x))\|
 &=\sup_{\iota\in I}\sup_{x\in\omega}\|T^{E}S(\mathsf{A}_{\iota})(x)\|\nu(x)
 =\sup_{\iota\in I}\sup_{x\in\omega}\|T^{E}f_{\iota}(x)\|\nu(x)\\
 &=\sup_{\iota\in I}|f_{\iota}|_{\FVE}<\infty
\end{align*}
by consistency. Further, $\lim_{\iota}S(\mathsf{A}_{\iota})(T^{\K}_{x})=\lim_{\iota}T^E(f_{\iota})(x)$
exists for each $x\in U$, implying the existence of $\lim_{\iota}S(\mathsf{A}_{\iota})(x)$ for each $x\in X$ 
by linearity. 
We apply \prettyref{prop:Blaschke_operator} and obtain $f:=S(\mathsf{A})\in\Feps$ 
such that $(\mathsf{A}_{\iota})_{\iota\in I}$ converges to $\mathsf{A}$ in $\F\varepsilon E$. 
From $\F$ and $\FE$ being $\varepsilon$-into-compatible it follows that 
$(f_{\iota})_{\iota\in I}$ converges to $f$ in $\FE$.
\end{proof}
 
First, we apply the preceding corollary $\gamma$-H\"older continuous functions.
Similar to $\mathcal{C}^{0,\gamma}(\overline{\Omega},E)$ we define the space of $E$-valued $\gamma$-H\"older continuous 
functions on $\Omega$ that vanish at a fixed point $z\in\Omega$, but with a different topology. 
Let $(\Omega,\d)$ be a metric space, $z\in\Omega$, $E$ an lcHs, $0<\gamma\leq 1$ and define
\[
\mathcal{C}^{[\gamma]}_{z}(\Omega,E):=\{f\in E^{\Omega}\;|\;f(z)=0\;\text{and}\; 
|f|_{\mathcal{C}^{0,\gamma}(\Omega),\alpha}<\infty\;\forall\;\alpha\in \mathfrak{A}\}.
\]
Further, we set $\omega:=\Omega^{2}\setminus\{(x,x)\;|\;x\in\Omega\}$, 
$\FE:=\{f\in\mathcal{C}(\Omega,E)\;|\;f(z)=0\}$ 
and $T^{E}\colon\FE\to E^{\omega}$, $T^{E}(f)(x,y):=f(x)-f(y)$, and 
\[
\nu\colon\omega\to [0,\infty),\;\nu(x,y):=\frac{1}{\d(x,y)^{\gamma}}.
\]
Then we have for every $\alpha\in\mathfrak{A}$ that
\[
 |f|_{\mathcal{C}^{0,\gamma}(\Omega),\alpha}=\sup_{x\in\omega}p_{\alpha}\bigl(T^{E}(f)(x)\bigr)\nu(x),
 \quad f\in\mathcal{C}^{[\gamma]}_{z}(\Omega,E),
\]
and observe that $\FVE=\mathcal{C}^{[\gamma]}_{z}(\Omega,E)$ with generator $(T^{E},T^{\K})$. 

\begin{cor}\label{cor:Hoelder_vanish_Blaschke}
Let $E$ be a Banach space, $(\Omega,\d)$ a metric space with finite diameter, $z\in\Omega$ and $0<\gamma\leq 1$.
If $(f_{\iota})_{\iota\in I}$ is a bounded net in $\mathcal{C}^{[\gamma]}_{z}(\Omega,E)$ such that 
$\lim_{\iota}f_{\iota}(x)$ exists for all $x$ in a dense subset $U\subset\Omega$,
then there is $f\in\mathcal{C}^{[\gamma]}_{z}(\Omega,E)$ such that 
$(f_{\iota})_{\iota\in I}$ converges to $f$ in $\mathcal{C}(\Omega,E)$ uniformly on 
compact subsets of $\Omega$.
\end{cor}
\begin{proof}
We take $\F:=\{f\in\mathcal{C}(\Omega)\;|\;f(z)=0\}$ 
and $\FE:=\{f\in\mathcal{C}(\Omega,E)\;|\;f(z)=0\}$. Then we have 
$\FV=\mathcal{C}^{[\gamma]}_{z}(\Omega)$ 
and $\FVE=\mathcal{C}^{[\gamma]}_{z}(\Omega,E)$ with the weight $\nu$ and generator $(T^{E},T^{\K})$ 
for $(\mathcal{F}\nu,E)$ described above.
Due to \cite[3.1 Bemerkung, p.\ 141]{B2} the spaces $\F$ and $\FE$, equipped 
with the compact-open topology $\tau_{co}$, 
are $\varepsilon$-compatible. Obviously, $(T^{E},T^{\K})$ is a strong, consistent family for $(\mathcal{F},E)$.
In addition, $\FV=\mathcal{C}^{[\gamma]}_{z}(\Omega)$ is a Banach space 
by \cite[Proposition 1.6.2, p.\ 20]{Weaver}. For all $f$ from the closed unit ball $B_{\FV}$ 
of $\FV$ we have
\[
|f(x)-f(y)|\leq \d(x,y)^{\gamma}\leq\operatorname{diam}(\Omega)^{\gamma},\quad x,y\in\Omega,
\]
where $\operatorname{diam}(\Omega):=\sup\{\d(x,y)\;|\;x,y\in\Omega\}$ is the finite diameter of $\Omega$. 
It follows that $B_{\FV}$ is (uniformly) equicontinuous and 
$\{f(x)\;|\;f\in B_{\FV}\}$ is bounded in $\K$ for all $x\in\Omega$. 
Ascoli's theorem (see e.g.\ \cite[Theorem 47.1, p.\ 290]{munkres2000}) implies the compactness of 
$B_{\FV}$ in $\F$.
Furthermore, the $\varepsilon$-compatibility of $\F$ and $\FE$ 
in combination with the consistency of $(T^{E},T^{\K})$ for $(\mathcal{F},E)$ gives $\Feps=\FVE$
as linear spaces by \prettyref{prop:mingle-mangle} c). 
We note that $\lim_{\iota}f_{\iota}(x)=\lim_{\iota}T^{E}(f_{\iota})(x,z)$ for all $x$ in $U$,
proving our claim by \prettyref{cor:Blaschke_vector_valued}.
\end{proof}

The space $\mathcal{C}^{[\gamma]}_{z}(\Omega)$ is named $\operatorname{Lip}_{0}(\Omega^{\gamma})$ 
in \cite{Weaver} (see \cite[Definition 1.6.1 (b), p.\ 19]{Weaver} and \cite[Definition 1.1.2, p.\ 2]{Weaver}). 
\prettyref{cor:Hoelder_vanish_Blaschke} generalises \cite[Proposition 2.1.7, p.\ 38]{Weaver} 
(in combination with \cite[Proposition 1.2.4, p.\ 5]{Weaver}) where $\Omega$ is compact, $U=\Omega$ and $E=\K$. 

\begin{cor}\label{cor:Hoelder_Blaschke}
Let $E$ be a Banach space, $\Omega\subset\R^{d}$ open and bounded, $k\in\N_{0}$ and $0<\gamma\leq 1$. 
In the case $k\geq 1$, assume additionally that $\Omega$ has Lipschitz boundary. 
If $(f_{\iota})_{\iota\in I}$ is a bounded net in $\mathcal{C}^{k,\gamma}(\overline{\Omega},E)$ such that 
\begin{enumerate}
\item [(i)] $\lim_{\iota}f_{\iota}(x)$ exists for all $x$ in a dense subset $U\subset\Omega$, or if
\item [(ii)] $\lim_{\iota}(\partial^{e_{n}})^{E}f_{\iota}(x)$ exists for all $1\leq n\leq d$ and $x$ in 
a dense subset $U\subset\Omega$, $\Omega$ is connected and there is $x_{0}\in\overline{\Omega}$ such that 
$\lim_{\iota}f_{\iota}(x_{0})$ exists and $k\geq 1$, 
\end{enumerate}
then there is $f\in\mathcal{C}^{k,\gamma}(\overline{\Omega},E)$ such that 
$(f_{\iota})_{\iota\in I}$ converges to $f$ in $\mathcal{C}^{k}(\overline{\Omega},E)$.
\end{cor}
\begin{proof}
As in \prettyref{cor:ext_B_unique} we choose $\F:=\mathcal{C}^{k}(\overline{\Omega})$ 
and $\FE:=\mathcal{C}^{k}(\overline{\Omega},E)$ 
as well as $\FV:=\mathcal{C}^{k,\gamma}(\overline{\Omega})$ 
and $\FVE:=\mathcal{C}^{k,\gamma}(\overline{\Omega},E)$ with the weight $\nu$ and generator $(T^{E},T^{\K})$ 
for $(\mathcal{F}\nu,E)$ described above of \prettyref{cor:ext_B_unique}.
By the proof of \prettyref{cor:ext_B_unique} all conditions of \prettyref{cor:Blaschke_vector_valued} are satisfied, 
which implies our statement.
\end{proof}

\begin{cor}\label{cor:loc_Hoelder_Blaschke}
Let $E$ be a Banach space, $\Omega\subset\R^{d}$ open, $k\in\N_{0}$ and $0<\gamma\leq 1$.
If $(f_{\iota})_{\iota\in I}$ is a bounded net in $\mathcal{C}^{k,\gamma}_{loc}(\Omega,E)$ such that 
\begin{enumerate}
\item [(i)] $\lim_{\iota}f_{\iota}(x)$ exists for all $x$ in a dense subset $U\subset\Omega$, or if
\item [(ii)] $\lim_{\iota}(\partial^{e_{n}})^{E}f_{\iota}(x)$ exists for all $1\leq n\leq d$ and $x$ in 
a dense subset $U\subset\Omega$, $\Omega$ is connected and there is $x_{0}\in\Omega$ such that 
$\lim_{\iota}f_{\iota}(x_{0})$ exists and $k\geq 1$, 
\end{enumerate}
then there is $f\in\mathcal{C}^{k,\gamma}_{loc}(\Omega,E)$ such that 
$(f_{\iota})_{\iota\in I}$ converges to $f$ in $\mathcal{C}^{k}(\Omega,E)$.
\end{cor}
\begin{proof}
Let $(\Omega_{n})_{n\in\N}$ be an exhaustion of $\Omega$ with open, relatively compact sets $\Omega_{n}\subset\Omega$ 
such that $\Omega_{n}$ has Lipschitz boundary, $\Omega_{n}\subset\Omega_{n+1}$ for all $n\in\N$ and, 
in addition, $x_{0}\in\Omega_{1}$ and $\Omega_{n}$ is connected for each $n\in\N$ in case (ii) 
(see the proof of \prettyref{cor:ext_B_unique_loc_Hoelder}).
The restriction of $(f_{\iota})_{\iota\in I}$ to $\Omega_{n}$ is a bounded net in 
$\mathcal{C}^{k,\gamma}(\overline{\Omega}_{n},E)$ for each $n\in\N$. By \prettyref{cor:Hoelder_Blaschke} there 
is $F_{n}\in\mathcal{C}^{k,\gamma}(\overline{\Omega}_{n},E)$ for each $n\in\N$ such that 
the restriction of $(f_{\iota})_{\iota\in I}$ to $\Omega_{n}$ converges to $F_{n}$ 
in $\mathcal{C}^{k}(\overline{\Omega}_{n},E)$ since $U\cap\Omega_{n}$ is dense in $\Omega_{n}$ 
due to $\Omega_{n}$ being open and $x_{0}$ being an element of the connected set $\Omega_{n}$ in case (ii). 
The limits $F_{n+1}$ and $F_{n}$ coincide on $\Omega_{n}$ for each $n\in\N$. Thus the 
definition $f:=F_{n}$ on $\Omega_{n}$ for each $n\in\N$ gives a well-defined 
function $f\in\mathcal{C}^{k,\gamma}_{loc}(\Omega,E)$, 
which is a limit of $(f_{\iota})_{\iota\in I}$ in $\mathcal{C}^{k}(\Omega,E)$.
\end{proof}

\begin{cor}\label{cor:k+1_smooth_Blaschke}
Let $E$ be a Banach space, $\Omega\subset\R^{d}$ open and $k\in\N_{0}$.
If $(f_{\iota})_{\iota\in I}$ is a bounded net in $\mathcal{C}^{k+1}(\Omega,E)$ such that 
\begin{enumerate}
\item [(i)] $\lim_{\iota}f_{\iota}(x)$ exists for all $x$ in a dense subset $U\subset\Omega$, or if
\item [(ii)] $\lim_{\iota}(\partial^{e_{n}})^{E}f_{\iota}(x)$ exists for all $1\leq n\leq d$ and $x$ in 
a dense subset $U\subset\Omega$, $\Omega$ is connected and there is $x_{0}\in\Omega$ such that 
$\lim_{\iota}f_{\iota}(x_{0})$ exists, 
\end{enumerate}
then there is $f\in\mathcal{C}^{k,1}_{loc}(\Omega,E)$ such that 
$(f_{\iota})_{\iota\in I}$ converges to $f$ in $\mathcal{C}^{k}(\Omega,E)$.
\end{cor}
\begin{proof}
By \prettyref{cor:ext_B_unique_loc_Hoelder} b) $(f_{\iota})_{\iota\in I}$ is a bounded net 
in $\mathcal{C}^{k,1}_{loc}(\Omega,E)$. 
Hence our statement is a consequence of \prettyref{cor:loc_Hoelder_Blaschke}.
\end{proof}

The preceding result directly implies a $\mathcal{C}^{\infty}$-smooth version.

\begin{cor}
Let $E$ be a Banach space and $\Omega\subset\R^{d}$ open.
If $(f_{\iota})_{\iota\in I}$ is a bounded net in $\mathcal{C}^{\infty}(\Omega,E)$ such that 
\begin{enumerate}
\item [(i)] $\lim_{\iota}f_{\iota}(x)$ exists for all $x$ in a dense subset $U\subset\Omega$, or if
\item [(ii)] $\lim_{\iota}(\partial^{e_{n}})^{E}f_{\iota}(x)$ exists for all $1\leq n\leq d$ and $x$ in 
a dense subset $U\subset\Omega$, $\Omega$ is connected and there is $x_{0}\in\Omega$ such that 
$\lim_{\iota}f_{\iota}(x_{0})$ exists, 
\end{enumerate}
then there is $f\in\mathcal{C}^{\infty}(\Omega,E)$ such that 
$(f_{\iota})_{\iota\in I}$ converges to $f$ in $\mathcal{C}^{\infty}(\Omega,E)$.
\end{cor}

Now, we turn to weighted kernels of hypoelliptic linear partial differential operators.

\begin{cor}\label{cor:weighted_hypo_Blaschke}
Let $E$ be a Banach space, $\Omega\subset\R^{d}$ open, 
$P(\partial)^{\K}$ a hypoelliptic linear partial differential operator, 
$\nu\colon\Omega\to(0,\infty)$ continuous and $U\subset\Omega$ a set of uniqueness for 
$(\id_{\K^{\Omega}},\mathcal{C}\nu^{\infty}_{P(\partial)})$. 
If $(f_{\iota})_{\iota\in I}$ is a bounded net in $(\mathcal{C}\nu^{\infty}_{P(\partial)}(\Omega,E),|\cdot|_{\nu})$ 
such that $\lim_{\iota}f_{\iota}(x)$ exists for all $x\in U$, 
then there is $f\in\mathcal{C}\nu^{\infty}_{P(\partial)}(\Omega,E)$ such that 
$(f_{\iota})_{\iota\in I}$ converges to $f$ in $(\mathcal{C}^{\infty}_{P(\partial)}(\Omega,E),\tau_{co})$.
\end{cor}
\begin{proof}
Our statement follows from \prettyref{cor:Blaschke_vector_valued} since by the proof of 
\prettyref{cor:hypo_weighted_ext_unique} all conditions needed are fulfilled. 
\end{proof}

For $\nu=1$ on $\Omega$ the preceding corollary is included in \cite[Corollary 4.2, p.\ 695]{F/J/W} but then 
an even better result is available, whose proof we prepare next.
For an open set $\Omega\subset\R^{d}$, an lcHs $E$ and a linear partial differential operator 
$P(\partial)^{E}\colon\mathcal{C}^{\infty}(\Omega,E)\to\mathcal{C}^{\infty}(\Omega,E)$ 
which is hypoelliptic if $E=\K$ we define the space of bounded zero solutions 
\[
 \mathcal{C}^{\infty}_{P(\partial),b}(\Omega,E):
=\{f\in\mathcal{C}^{\infty}_{P(\partial)}(\Omega,E)\;|\;\forall\;\alpha\in\mathfrak{A}:\;
   \|f\|_{\infty,\alpha}:=\sup_{x\in\Omega}p_{\alpha}(f(x))<\infty\}.
\]
Apart from the topology given by $(\|\cdot\|_{\infty,\alpha})_{\alpha\in\mathfrak{A}}$ 
there is another weighted locally convex topology 
on $\mathcal{C}^{\infty}_{P(\partial),b}(\Omega,E)$ which is of interest, namely, the one induced by the seminorms 
\[
|f|_{\widetilde{\nu},\alpha}:=\sup_{x\in\Omega}p_{\alpha}(f(x))|\widetilde{\nu}(x)|,\quad f\in\mathcal{C}^{\infty}_{P(\partial),b}(\Omega,E),
\]
for $\widetilde{\nu}\in\mathcal{C}_{0}(\Omega)$ and $\alpha\in\mathfrak{A}$.
We denote by $(\mathcal{C}^{\infty}_{P(\partial),b}(\Omega,E),\beta)$ 
the space $\mathcal{C}^{\infty}_{P(\partial),b}(\Omega,E)$ 
equipped with the strict topology $\beta$ induced by the seminorms 
$(|\cdot|_{\widetilde{\nu},\alpha})_{\widetilde{\nu}\in\mathcal{C}_{0}(\Omega),\alpha\in\mathfrak{A}}$.\label{page:strict}
Now, we phrase for $\mathcal{C}^{\infty}_{P(\partial),b}(\Omega,E)=\mathcal{C}\nu^{\infty}_{P(\partial)}(\Omega,E)$ 
with $\nu=1$ on $\Omega$ the improved version of \prettyref{cor:weighted_hypo_Blaschke}.

\begin{cor}\label{cor:strict_hypo_Blaschke}
Let $E$ be a Banach space, $\Omega\subset\R^{d}$ open, $P(\partial)^{\K}$ a hypoelliptic linear 
partial differential operator and $U\subset\Omega$ a set of uniqueness 
for $(\id_{\K^{\Omega}},\mathcal{C}^{\infty}_{P(\partial),b})$. 
If $(f_{\iota})_{\iota\in I}$ is a bounded net in $(\mathcal{C}^{\infty}_{P(\partial),b}(\Omega,E),\|\cdot\|_{\infty})$ 
such that $\lim_{\iota}f_{\iota}(x)$ exists for all $x\in U$, 
then there is $f\in\mathcal{C}^{\infty}_{P(\partial),b}(\Omega,E)$ such that 
$(f_{\iota})_{\iota\in I}$ converges to $f$ in $(\mathcal{C}^{\infty}_{P(\partial),b}(\Omega,E),\beta)$.
\end{cor}
\begin{proof}
We take $\F:=(\mathcal{C}^{\infty}_{P(\partial),b}(\Omega),\beta)$ 
and $\FE:=(\mathcal{C}^{\infty}_{P(\partial),b}(\Omega,E),\beta)$ 
as well as $\FV:=(\mathcal{C}^{\infty}_{P(\partial),b}(\Omega),\|\cdot\|_{\infty})$ 
and $\FVE:=(\mathcal{C}^{\infty}_{P(\partial),b}(\Omega,E),\|\cdot\|_{\infty})$ 
with the weight $\nu(x):=1$, $x\in\Omega$, and generator $(\id_{E^{\Omega}},\id_{\Omega^{\K}})$ 
for $(\mathcal{F}\nu,E)$. The generator is strong and consistent for $(\mathcal{F},E)$ 
and $\F$ and $\FE$ are $\varepsilon$-compatible by \cite[3.14 Proposition, p.\ 12]{kruse2018_3}. 
The space $\FV$ is a Banach space as a closed subspace of the Banach space 
$(\mathcal{C}_{b}(\Omega),\|\cdot\|_{\infty})$. Its closed unit ball $B_{\FV}$ is $\tau_{co}$-compact because 
$(\mathcal{C}^{\infty}_{P(\partial)}(\Omega),\tau_{co})$ is a Fr\'echet--Schwartz space, in particular a Montel space. 
Thus $B_{\FV}$ is $\|\cdot\|_{\infty}$-bounded and $\tau_{co}$-compact, which implies that 
it is also $\beta$-compact by \cite[Proposition 1 (viii), p.\ 586]{cooper1971} 
and \cite[Proposition 3, p.\ 590]{cooper1971}. 
In addition, the $\varepsilon$-compatibility of $\F$ and $\FE$ 
in combination with the consistency of $(\id_{E^{\Omega}},\id_{\K^{\Omega}})$ for $(\mathcal{F},E)$ gives $\Feps=\FVE$
as linear spaces by \prettyref{prop:mingle-mangle} c), verifying our statement 
by \prettyref{cor:Blaschke_vector_valued}.
\end{proof}

A direct consequence of \prettyref{cor:strict_hypo_Blaschke} is the following remark.

\begin{rem}
Let $E$ be a Banach space, $\Omega\subset\R^{d}$ open, $P(\partial)^{\K}$ a hypoelliptic linear partial 
differential operator and $(f_{\iota})_{\iota\in I}$ a bounded net in 
$(\mathcal{C}^{\infty}_{P(\partial),b}(\Omega,E),\|\cdot\|_{\infty})$. 
Then the following statements are equivalent:
\begin{enumerate} 
\item[(i)] $(f_{\iota})$ converges pointwise,
\item[(ii)] $(f_{\iota})$ converges uniformly on compact subsets of $\Omega$,
\item[(iii)] $(f_{\iota})$ is $\beta$-convergent.  
\end{enumerate}
\end{rem}

In the case of complex-valued bounded holomorphic functions of one variable, i.e.\ $E=\C$, $\Omega\subset\C$ is open and 
$P(\partial)=\overline{\partial}$ is the Cauchy-Riemann operator, convergence w.r.t.\ $\beta$ is known as 
bounded convergence (see \cite[p.\ 13--14, 16]{rubel1971}) and
the preceding remark is included in \cite[3.7 Theorem, p.\ 246]{rubelshields1966} for connected sets $\Omega$.

A similar improvement of \prettyref{cor:Hoelder_vanish_Blaschke} for the space 
$\mathcal{C}^{[\gamma]}_{z}(\Omega,E)$ of $\gamma$-H\"older continuous functions on a metric space $(\Omega,\d)$ 
that vanish at a given point $z\in\Omega$ is possible, using the strict topology $\beta$ on 
$\mathcal{C}^{[\gamma]}_{z}(\Omega)$ given by the seminorms
\[
|f|_{\nu}:=\sup_{\substack{x,y\in\Omega\\x\neq y}}
\frac{|f(x)-f(y)|}{|x-y|^{\gamma}}|\nu(x,y)|,\quad f\in\mathcal{C}^{[\gamma]}_{z}(\Omega),
\]
for $\nu\in\mathcal{C}_{0}(\omega)$ with $\omega=\Omega^{2}\setminus\{(x,x)\;|\;x\in\Omega\}$. 
If $\Omega$ is compact and $E$ a Banach space, this follows as in \prettyref{cor:strict_hypo_Blaschke} 
from the observation that $\beta$ is the mixed topology 
$\gamma=\gamma(|\cdot|_{\mathcal{C}^{0,\gamma}(\Omega)},\tau_{co})$ 
by \cite[Theorem 3.3, p.\ 645]{vargas2018}, that a set is $\beta$-compact if and only if it 
is $|\cdot|_{\mathcal{C}^{0,\gamma}(\Omega)}$-bounded and $\tau_{co}$-compact 
by \cite[Theorem 2.1 (6), p.\ 642]{vargas2018}, the $\varepsilon$-compatibility
$(\mathcal{C}^{[\gamma]}_{z}(\Omega),\beta)\varepsilon E\cong(\mathcal{C}^{[\gamma]}_{z}(\Omega,E),\gamma\tau_{\gamma})$ 
by \cite[Theorem 4.4, p.\ 648]{vargas2018} where the topology 
$\gamma\tau_{\gamma}$ is described in \cite[Definition 4.1, p.\ 647]{vargas2018} 
and coincides with $\beta$ if $E=\K$ by \cite[Proposition 4.3 (i), p.\ 647]{vargas2018}.

Let us turn to Bloch type spaces. The result corresponding to \prettyref{cor:weighted_hypo_Blaschke} 
for Bloch type spaces reads as follows.

\begin{cor}
Let $E$ be a Banach space, $\nu\colon\D\to(0,\infty)$ continuous and $U_{\ast}\subset\D$ 
have an accumulation point in $\D$. 
If $(f_{\iota})_{\iota\in I}$ is a bounded net in $\mathcal{B}\nu(\D,E)$ 
such that $\lim_{\iota}f_{\iota}(0)$ and $\lim_{\iota}(\partial_{\C}^{1})^{E}f_{\iota}(z)$ 
exist for all $z\in U_{\ast}$, 
then there is $f\in\mathcal{B}\nu(\D,E)$ such that 
$(f_{\iota})_{\iota\in I}$ converges to $f$ in $(\mathcal{O}(\D,E),\tau_{co})$. 
\end{cor}
\begin{proof}
Due to the proof of \prettyref{cor:Bloch_ext_unique} all conditions needed 
to apply \prettyref{cor:Blaschke_vector_valued} are fulfilled, which proves our statement.
\end{proof}
\section{Wolff type results}
The following theorem gives us a Wolff type description of the dual of $\F$ 
and generalises \cite[Theorem 3.3, p.\ 693]{F/J/W} and its \cite[Corollary 3.4, p.\ 694]{F/J/W} 
whose proofs only need a bit of adaptation. 

\begin{thm}\label{thm:wolff}
Let $\F$ and $\FE$ be $\varepsilon$-into-compatible, 
$(T^{E},T^{\K})$ be a generator for $(\mathcal{F}\nu,E)$ and a strong, consistent family for $(\mathcal{F},E)$ 
for every Banach space $E$. 
Let $\F$ be a nuclear Fr\'echet space and $\FV$ a Banach space 
whose closed unit ball $B_{\FV}$ 
is a compact subset of $\F$ and $(x_{n})_{n\in\N}$ fixes the topology in $\FV$.
\begin{enumerate}
\item[a)] Then there is $0<\lambda\in\ell^1$ such that for every bounded $B\subset\F_{b}'$ there 
is $C\geq 1$ with 
\[
       \{\mu_{\mid\FV}\;|\;\mu\in B\}
\subset\{\sum_{n=1}^{\infty}a_{n}\nu(x_{n})T^{\K}_{x_{n}}\in\FV'
         \;|\;a\in\ell^{1},\;\forall\;n\in\N:\;|a_{n}|\leq C\lambda_{n}\}.
\]
\item[b)] Let $(\|\cdot\|_{k})_{k\in\N}$ denote the system of seminorms generating the topology of $\F$. 
Then there is a decreasing zero sequence $(\varepsilon_{n})_{n\in\N}$ such that for all $k\in\N$ 
there is $C\geq 1$ with 
\[
\|f\|_{k}\leq C \sup_{n\in\N}|T^{\K}(f)(x_{n})|\nu(x_{n})\varepsilon_{n},\quad f\in\FV. 
\]
\end{enumerate}
\end{thm}
\begin{proof}
We start with part a). Let $B_{1}:=\{T^{\K}_{x_{n}}(\cdot)\nu(x_{n})\;|\;n\in\N\}\subset\F'$, 
$X:=\operatorname{span} B_{1}$, $Y:=\F$, 
$Z:=\FV$ and $E_{1}:=\{\sum_{n=1}^{\infty}a_{n}\nu(x_{n})T^{\K}_{x_{n}}\;|\;a\in\ell^{1}\}$. 
From
\[
|j_{1}(a)(f)|:=|\sum_{n=1}^{\infty}a_{n}\nu(x_{n})T^{\K}_{x_{n}}(f)|\leq\sup_{n\in\N}|T^{\K}(f)(x_{n})|\nu(x_{n})\|a\|_{\ell^{1}}
\leq |f|_{\FV}\|a\|_{\ell^{1}}
\]
for all $f\in\FV$ and $a\in\ell^{1}$ it follows that $E_{1}$ is a linear subspace of $\FV'$ 
and the continuity of the map $j_{1}\colon \ell^{1}\to\FV'$ where $\FV'$ is equipped 
with the operator norm. In addition, we deduce that the linear map 
$j\colon \ell^{1}/\ker j_{1}\to\FV'$, $j([a]):=j_{1}(a)$, where 
$[a]$ denotes the equivalence class of $a\in\ell^{1}$ in the quotient space $\ell^{1}/\ker j_{1}$,
is continuous w.r.t.\ the quotient norm since
\[
\|j([a])\|_{\FV'}\leq \inf_{b\in\ell^{1},[b]=[a]}\|b\|_{\ell^{1}}=\|[a]\|_{\ell^{1}/\ker j_{1}}.
\]
By setting $E:=j(\ell^{1}/\ker j_{1})$ and $\|j([a])\|_{E}:=\|[a]\|_{\ell^{1}/\ker j_{1}}$, $a\in\ell^{1}$, 
and observing that $\ell^{1}/\ker j_{1}$ is a Banach space, we obtain that $E$ is also a Banach space, 
which is continuously embedded in $\FV'$.

We denote by $\mathsf{A}\colon X\to E$ the restriction to $Z=\FV$ determined by 
\[
\mathsf{A}(T^{\K}_{x_{n}}(\cdot)\nu(x_{n})):=T^{\K}_{x_{n}}(\cdot)_{\mid\FV}\nu(x_{n})=j([e_{n}])
\]
where $e_{n}$ is the $n$-th unit sequence in $\ell^{1}$. We consider $\FV$ as a subspace 
of $E'$ via 
\[
f(j([a])):=j([a])(f)=\sum_{n=1}^{\infty}a_{n}\nu(x_{n})T^{\K}(f)(x_{n}),\quad a\in\ell^{1},
\]
for $f\in\FV$. The space $G:=\FV$ clearly separates the points of $E$, 
thus is $\sigma(E',E)$-dense and  
\[
 (f\circ \mathsf{A})(T^{\K}_{x_{n}}(\cdot)\nu(x_{n})) 
=\mathsf{A}(T^{\K}_{x_{n}}(\cdot)\nu(x_{n}))(f)
=j([e_{n}])(f)
=f(j([e_{n}]))
\]
for all $n\in\N$. Hence we may consider $f\circ \mathsf{A}$ by identification with $f$ as an element of 
$Z=\FV$ for all $f\in G=\FV$. It follows from \prettyref{prop:ext_B_fix_top} that there is 
a unique extension $\widehat{\mathsf{A}}\in\F\varepsilon E$ of $\mathsf{A}$ such that 
$S(\widehat{\mathsf{A}})\in\Feps$. 

For each $e'\in E'$ there are $C_{0},C_{1}>0$ and an absolutely convex compact set $K\subset\F$ such that
\[
|(e'\circ\widehat{\mathsf{A}})(\mu)|\leq C_{0}\|\widehat{\mathsf{A}}(\mu)\|_{E}\leq C_{0}C_{1}\sup_{f\in K}|\mu(f)|
\]
for all $\mu\in\F'$, implying $e'\circ \widehat{\mathsf{A}}\in(\F_{b}')'$. 
Due to the reflexivity of the nuclear Fr\'echet space $\F$ we obtain 
$e'\circ \widehat{\mathsf{A}}\in\F$ for each $e'\in E'$. 
Further, for each $e'\in E'$ we have 
\begin{align*}
  \|e'\circ\widehat{\mathsf{A}}\|_{\FV}
&=\sup_{x\in\omega}|T^{\K}(e'\circ\widehat{\mathsf{A}})(x)|\nu(x)
 =\sup_{x\in\omega}|(e'\circ\widehat{\mathsf{A}})(T^{\K}_{x}(\cdot)\nu(x))|\\
&\leq C_{0}\sup_{x\in\omega}\|\widehat{\mathsf{A}}(T^{\K}_{x}(\cdot)\nu(x))\|_{E}<\infty
\end{align*}
since $\widehat{\mathsf{A}}(B_{\FV}^{\circ \F'})$ is bounded in $E$. 
This yields $e'\circ\widehat{\mathsf{A}}\in\FV$ for each $e'\in E'$. In particular, we get that
%\[
%(e'\circ\widehat{\mathsf{A}})(\mu)=\mu(e'\circ\widehat{\mathsf{A}})
%\]
$\widehat{\mathsf{A}}$ is $\sigma(\F',\FV)$-$\sigma(E,E')$ continuous. 
The restriction $r\colon\F'\to \FV'$, $r(\mu):=\mu_{\mid \FV}$, 
is $\sigma(\F',\FV)$-$\sigma(\FV',\FV)$ continuous 
and coincides with $\widehat{\mathsf{A}}$ on the $\sigma(\F',\FV)$-dense subspace 
$X=\operatorname{span} B_{1}\subset\F'$. 
Therefore $\widehat{\mathsf{A}}(\mu)=r(\mu)=\mu_{\mid \FV}$ for all $\mu\in\F'$.

Let $B$ be an absolutely convex, closed and bounded subset of $\F_{b}'$. 
We endow $W:=\operatorname{span} B$ with the Minkowski functional of $B$. 
Due to the nuclearity of $\F$, there are an absolutely convex, closed and bounded subset 
$V\subset\F_{b}'$, $(w_{k}')_{k\in\N}\subset B_{W'}$, $(\mu_{k})_{k\in\N}\subset V$ 
and $0\leq \gamma\in\ell^{1}$ such that
\[
\mu=\sum_{k=1}^{\infty}\gamma_{k}w_{k}'(\mu)\mu_{k},\quad \mu\in B,
\]
by \cite[2.9.1 Theorem, p.\ 134, 2.9.2 Definition, p.\ 135]{bogachev2017}. 
The boundedness of $\widehat{\mathsf{A}}(V)$ in $E$ and the definition of $E$ give us a bounded sequence $([\beta^{(k)}])_{k\in\N}\subset E$ 
with 
\[
{\mu_{k}}_{\mid \FV}=\widehat{\mathsf{A}}(\mu_{k})=\sum_{n=1}^{\infty}\beta^{(k)}_{n}\nu(x_{n})T^{\K}_{x_{n}}
\]
for all $k\in\N$. The sequence $(\beta^{(k)})_{k\in\N}\subset\ell^{1}$ is also bounded by 
\cite[Remark 5.11, p.\ 36]{meisevogt1997} 
and we set $\rho_{n}:=\sum_{k=1}^{\infty}\gamma_{k}|\beta^{(k)}_{n}|$ for $n\in\N$. 
With $\rho:=(\rho_{n})_{n\in\N}$ we have
\[
\|\rho\|_{\ell^{1}}=\sum_{n=1}^{\infty}\sum_{k=1}^{\infty}\gamma_{k}|\beta^{(k)}_{n}|
\leq \sum_{n=1}^{\infty}\sup_{l\in\N}|\beta^{(l)}_{n}|\sum_{k=1}^{\infty}\gamma_{k}
=\sup_{l\in\N}\|\beta^{(l)}\|_{\ell^{1}}\|\gamma\|_{\ell^{1}}<\infty,
\]
which means that $\rho\in\ell^{1}$. 
For every $\mu\in B$ we set $a_{n}:=\sum_{k=1}^{\infty}\gamma_{k}w_{k}'(\mu)\beta^{(k)}_{n}$, $n\in\N$, 
and conclude that $a\in\ell^{1}$ with $|a_{n}|\leq\rho_{n}$ for all $n\in\N$ and
\begin{equation}\label{eq:wolff}
\mu_{\mid \FV}=\sum_{n=1}^{\infty}a_{n}\nu(x_{n})T^{\K}_{x_{n}}.
\end{equation}
The strong dual $\F_{b}'$ of the Fr\'echet--Schwartz space $\F$ is a DFS-space 
and thus there is a fundamental sequence of bounded sets $(B_{l})_{l\in\N}$ in $\F_{b}'$ 
by \cite[Proposition 25.19, p.\ 303]{meisevogt1997}. Due to our preceding results there is $\rho^{(l)}\in\ell^{1}$ 
with \eqref{eq:wolff} for each $l\in\N$. Finally, part a) follows from choosing $0<\lambda\in\ell^{1}$ such that 
each $\rho^{(l)}$ is componentwise smaller than a multiple of $\lambda$, i.e.\ we choose $\lambda$ in a way that 
there is $C_{l}\geq 1$ with $\rho^{(l)}_{n}\leq C_{l}\lambda_{n}$ for all $n\in\N$.

Let us turn to part b). We choose $\lambda\in\ell^{1}$ from part a) and a decreasing zero sequence 
$(\varepsilon_{n})_{n\in\N}$ such that $(\tfrac{\lambda_{n}}{\varepsilon_{n}})_{n\in\N}$ still belongs to $\ell^{1}$.
For $k\in\N$ we set 
\[
\widetilde{B}_{k}:=\{f\in\F\;|\;\|f\|_{k}\leq 1\}
\]
and note that the polar $\widetilde{B}_{k}^{\circ}$ is bounded in $\F_{b}'$. 
Due to part a) there exists $C\geq 1$ such that 
\[
\widehat{\mathsf{A}}(\widetilde{B}_{k}^{\circ})
\subset
\{\sum_{n=1}^{\infty}a_{n}\nu(x_{n})T^{\K}_{x_{n}}\in\FV'
         \;|\;a\in\ell^{1},\;\forall\;n\in\N:\;|a_{n}|\leq C\lambda_{n}\}.
\]
By \cite[Proposition 22.14, p.\ 256]{meisevogt1997} the formula 
\[
\|f\|_{k}=\sup_{y'\in \widetilde{B}_{k}^{\circ}}|y'(f)|,\quad f\in\F,
\]
is valid and hence 
\begin{align*}
  \|f\|_{k}
&=\sup_{y'\in \widetilde{B}_{k}^{\circ}}|r(y')(f)|
 =\sup_{y'\in \widetilde{B}_{k}^{\circ}}|\widehat{\mathsf{A}}(y')(f)|
 \leq C\sup_{\substack{a\in\ell^{1}\\ |a_{n}|\leq \lambda_{n}}}
      \bigl|\sum_{n=1}^{\infty}a_{n}\nu(x_{n})T^{\K}(f)(x_{n})\bigr|\\
&\leq C\Bigl\|\Bigl(\frac{\lambda_{n}}{\varepsilon_{n}}\Bigr)_{n}\Bigr\|_{\ell^{1}}
       \sup_{n\in\N}|T^{\K}(f)(x_{n})|\nu(x_{n})\varepsilon_{n}
\end{align*}
for all $f\in\FV$.
\end{proof}

\begin{rem}
The proof of \prettyref{thm:wolff} shows it is not needed that the assumption 
that $\F$ and $\FE$ are $\varepsilon$-into-compatible, 
$(T^{E},T^{\K})$ is a generator for $(\mathcal{F}\nu,E)$ and a strong, consistent family for $(\mathcal{F},E)$ 
is fulfilled for every Banach space $E$. It is sufficient that it is fulfilled for the Banach space 
$E:=j(\ell^{1}/\ker j_{1})$.
\end{rem}

We recall from \eqref{eq:frame} that for a positive sequence $\nu:=(\nu_{n})_{n\in\N}$ and an lcHs $E$ we have
\[
\ell\nu(\N,E)=\{x=(x_{n})_{n\in\N}\in E^{\N}\;|\;\forall\;\alpha\in\mathfrak{A}:\;
\|x\|_{\alpha}=\sup_{n\in\N}p_{\alpha}(x_{n})\nu_{n}<\infty\}.
\]
Further, we equip the space $E^{\N}$ of all sequences in $E$ with the topology of pointwise convergence, i.e.\ the 
topology generated by the seminorms 
\[
|x|_{k,\alpha}:=\sup_{1\leq n\leq k}p_{\alpha}(x_{n}),\quad x=(x_{n})_{n\in\N}\in E^{\N},
\]
for $k\in\N$ and $\alpha\in\mathfrak{A}$.

\begin{cor}
Let $\nu:=(\nu_{n})_{n\in\N}$ be a positive sequence. 
\begin{enumerate}
\item[a)] Then there is $0<\lambda\in\ell^1$ such that for every bounded $B\subset(\K^{\N})_{b}'$ there 
is $C\geq 1$ with 
\[
       \{\mu_{\mid\ell\nu(\N)}\;|\;\mu\in B\}
\subset\{\sum_{n=1}^{\infty}a_{n}\nu_{n}\delta_{n}\in\ell\nu(\N)'
         \;|\;a\in\ell^{1},\;\forall\;n\in\N:\;|a_{n}|\leq C\lambda_{n}\}.
\]
\item[b)] Then there is a decreasing zero sequence $(\varepsilon_{n})_{n\in\N}$ such that for all $k\in\N$ 
there is $C\geq 1$ with 
\[
\sup_{1\leq n\leq k}|x_{n}|\leq C \sup_{n\in\N}|x_{n}|\nu_{n}\varepsilon_{n},\quad x=(x_{n})_{n\in\N}\in\ell\nu(\N). 
\]
\end{enumerate}
\end{cor}
\begin{proof}
We take $\mathcal{F}(\N):=\K^{\N}$ 
and $\mathcal{F}(\N,E):=E^{\N}$ 
as well as $\mathcal{F}\nu(\N):=\ell\nu(\N)$ and 
$\mathcal{F}\nu(\N,E):=\ell\nu(\N,E)$
where $(T^{E},T^{\K}):=(\operatorname{\id}_{E^{\N}},\operatorname{\id}_{\K^{\N}})$ is the generator 
for $(\mathcal{F}\nu,E)$.
We remark that $\mathcal{F}(\N)$ and $\mathcal{F}(\N,E)$ are $\varepsilon$-compatible 
and $(T^{E},T^{\K})$ is a strong, consistent family for $(\mathcal{F},E)$ by \cite[2.4 Theorem (2), p.\ 138--139]{B2}
for every Banach space $E$. Moreover, $\mathcal{F}\nu(\N)=\ell\nu(\N)$ is a Banach space by 
\cite[Lemma 27.1, p.\ 326]{meisevogt1997} since $\ell\nu(\N)=\lambda^{\infty}(A)$ 
with the K\"othe matrix $A:=(a_{n,k})_{n,k\in\N}$ given by $a_{n,k}:=\nu_{n}$ for all $n,k\in\N$. 
In addition, we have for every $k\in\N$
\[
\sup_{1\leq n\leq k}|x_{n}|\leq \sup_{1\leq n\leq k}\nu_{n}^{-1}|x|_{\nu} \leq \sup_{1\leq n\leq k}\nu_{n}^{-1},
\quad x=(x_{n})_{n\in\N}\in B_{\mathcal{F}\nu(\N)},
\]
which means that $B_{\mathcal{F}\nu(\N)}$ is bounded in $\mathcal{F}(\N)$.  
The space $\mathcal{F}(\N)=\K^{\N}$ is a nuclear Fr\'echet space and 
$B_{\mathcal{F}\nu(\N)}$ is obviously closed in $\K^{\N}$. 
Thus the bounded and closed set $B_{\mathcal{F}\nu(\N)}$ is compact in $\mathcal{F}(\N)$,
implying our statement by \prettyref{thm:wolff}.
\end{proof}

\begin{cor}
Let $\Omega\subset\R^{d}$ be open, $P(\partial)^{\K}$ a hypoelliptic linear partial differential operator, 
$\nu\colon\Omega\to(0,\infty)$ continuous and $(x_{n})_{n\in\N}$ fix the topology 
in $\mathcal{C}\nu^{\infty}_{P(\partial)}(\Omega)$.
\begin{enumerate}
\item[a)] Then there is $0<\lambda\in\ell^1$ such that for every bounded 
$B\subset(\mathcal{C}^{\infty}_{P(\partial)}(\Omega),\tau_{co})_{b}'$ there is $C\geq 1$ with 
\[
       \{\mu_{\mid\mathcal{C}\nu^{\infty}_{P(\partial)}(\Omega)}\;|\;\mu\in B\}
\subset\{\sum_{n=1}^{\infty}a_{n}\nu(x_{n})\delta_{x_{n}}\in\mathcal{C}\nu^{\infty}_{P(\partial)}(\Omega)'
         \;|\;a\in\ell^{1},\;\forall\;n\in\N:\;|a_{n}|\leq C\lambda_{n}\}.
\]
\item[b)] Then there is a decreasing zero sequence $(\varepsilon_{n})_{n\in\N}$ such that 
for all compact $K\subset\Omega$ there is $C\geq 1$ with 
\[
\sup_{x\in K}|f(x)|\leq C \sup_{n\in\N}|f(x_{n})|\nu(x_{n})\varepsilon_{n},
\quad f\in\mathcal{C}\nu^{\infty}_{P(\partial)}(\Omega). 
\]
\end{enumerate}
\end{cor}
\begin{proof}
Due to the proof of \prettyref{cor:hypo_weighted_ext_unique} and the observation that the space 
$\F=(\mathcal{C}^{\infty}_{P(\partial)}(\Omega),\tau_{co})$ is a nuclear Fr\'echet space 
all conditions of \prettyref{thm:wolff} are fulfilled, which yields our statement.
\end{proof}
\bibliography{biblio}
\bibliographystyle{plainnat}
\end{document}